\documentclass[12pt,reqno]{amsart}
\usepackage{amssymb}
\usepackage{amsmath}
\usepackage{amsthm}
\usepackage{eucal}
\usepackage{color}
\usepackage{amsfonts}
\usepackage[pdftex]{graphicx}
\usepackage{mathtools}
\usepackage[normalem]{ulem}
\usepackage[T1]{fontenc}

\newcommand{\R}{\mathbb R}

\newcommand{\Z}{\mathbb Z}

\renewcommand{\k}{\kappa}

\newcommand{\rr}{{\mathbb R}}

\newcommand{\fff}{{\mathbf F}}
\newcommand{\vv}{{\mathbf v}}
\newcommand{\xx}{{\mathcal X}}
\newcommand{\lip}{\text{\rm{Lip}}}
\newcommand{\eps}{\varepsilon}

\newcommand{\paren}[1]{\left(#1\right)}
\newcommand{\pd}[2]{\frac{\partial#1}{\partial#2}}
\newcommand{\pds}[1]{\partial_{#1}}
\newcommand{\od}[2]{\frac{d#1}{d#2}}
\newcommand{\norm}[2]{\left\|{#1}\right\|_{#2}}
\newcommand{\intr}{\int_{-\infty}^\infty}
\newcommand{\oo}{{\mathcal O_\rho}}

\newcommand{\dd}{{\mathcal D}}

\newcommand{\alev}{\quad\text{a.e.}}
\newcommand{\oxy}{\oo\times X\times Y}
\newcommand{\xzw}{(\xi,z,w)}
\newcommand{\xsmxs}{(x(s)-x(\sigma))}
\newcommand{\tint}{[-T,T]}
\newcommand{\tints}[1]{[-T_{#1},T_{#1}]}

\newcommand{\neigh}{\mathcal{N}}
\DeclareMathOperator{\sgn}{sgn}
\DeclareMathOperator{\supp}{supp}
\DeclareMathOperator{\essinf}{essinf}

\newtheorem{theorem}{Theorem}[section]
\newtheorem{proposition}[theorem]{Proposition}
\newtheorem{lemma}[theorem]{Lemma}
\newtheorem{corollary}[theorem]{Corollary}
\newtheorem{definition}[theorem]{Definition}

\newtheorem{claim}[theorem]{Claim}

\numberwithin{equation}{section}
\theoremstyle{remark}
\newtheorem{remark}[theorem]{Remark}
\begin{document}
\title[The Camassa-Holm equation]{Properties of solutions
to the \\
Camassa-Holm equation on the line in\\
 a class containing the peakons}
\author{Felipe Linares}
\address[F. Linares]{IMPA\\
Instituto Matem\'atica Pura e Aplicada\\
Estrada Dona Castorina 110\\
22460-320, Rio de Janeiro, RJ\\Brazil}
\email{linares@impa.br}

\author{Gustavo Ponce}
\address[G. Ponce]{Department  of Mathematics\\
University of California\\
Santa Barbara, CA 93106\\
USA.}
\email{ponce@math.ucsb.edu}

\author{Thomas C. Sideris}
\address[T. C. Sideris]{Department  of Mathematics\\
University of California\\
Santa Barbara, CA 93106\\
USA.}
\email{ sideris@math.ucsb.edu}

\begin{abstract} We shall study special  properties of solutions to the IVP associated to the Camassa-Holm equation
{on the line
related to the regularity and the decay of  solutions.}
The first  aim is to show how the regularity on the initial data is transferred to the corresponding solution
{in a class containing the ``peakon solutions"}. In particular, we shall show that the local regularity is similar to that exhibited by the solution of the inviscid Burger's equation with the same initial datum. The second goal is to prove that the decay results obtained in \cite{HMPZ} extend to the class of solutions considered here.
\end{abstract}

\thanks
{The first author was partially supported by CNPq and FAPERJ/Brazil.}

\dedicatory{Dedicated to Professor Nakao Hayashi on the occasion of his 60th birthday.}
\maketitle

\section{Introduction}

This work is concerned with the non-periodic Camassa-Holm (CH) equation
\begin{equation}\label{CH}
\partial_tu+\k \partial_xu +3 u \partial_xu-\partial_t \partial_x^2 u=2\partial_xu \partial_x^2u+ u\partial_x^3u, \hskip5pt \;t,\,x,\,\k\in\R.
\end{equation}

The CH equation \eqref{CH} was first noted by Fuchssteiner and Fokas \cite{FF} in their work on hereditary symmetries. Later, it was written explicitly and derived physically as a model for shallow water waves ($\kappa>0$) by Camassa and Holm \cite{CH}, who also studied its solutions. The CH equation \eqref{CH} has received considerable attention due to its remarkable properties, among them the fact that it is a bi-Hamiltonian completely integrable model  for all values of $k\in \R$, (see \cite{BSS},  \cite{CH}, \cite{CoMc}, \cite{Mc}, \cite{M}, \cite{Par} and references therein).

 By omitting the right hand side in \eqref{CH}, the CH equation reduces to the so called Benjamin-Bona-Mahony  equation \cite{BBM}, also deduced in the context of water waves. 

 The case $\k=0$ in \eqref{CH}, called the reduced Camassa-Holm RCH (see \cite{Par}), 
\begin{equation}\label{RCH}
\partial_tu+3 u \partial_xu-\partial_t \partial_x^2u=2\partial_xu \partial_x^2u+ u\partial_x^3u, \hskip5pt \;t,\,x\in\R,
\end{equation}
has motivated a great deal of research. 
It appears as a model in nonlinear dispersive waves in hyperelastic rods \cite{Dai}. The RCH equation possesses ``peakon'' solutions
\cite{CH}. In the case of a single peakon this solitary wave solution can be written as
\begin{equation}\label{peakon}
u_c(x,t)=c\, e^{-|x-ct|}, \hskip20pt c>0.
\end{equation}
The multi-peakon solutions display the  ``elastic'' collision property that reflect their soliton character. Thus, the CH equation and the Korteweg-de Vries  equation
\begin{equation}\label{KdV}
\partial_tu+ u \partial_xu+\partial_x^3u=0, \hskip10pt \;t,\,x\in\R.
\end{equation}
exhibit many features in common.

The initial value problem (IVP)  as well as the periodic boundary value problem associated to the equation \eqref{CH} has been extensively examined. 
In particular, in \cite{LO} and \cite{RB}  the local well-posedness (LWP) of the IVP was established in the Sobolev space
\[
H^s(\R)=(1-\partial_x^2)^{-s/2}L^2(\R),
\]
for $s>3/2$. The peakon solutions do not belong to these spaces,
see Corollary \ref{pro1}. However,
\[
\phi(x)=e^{-|x|}\in W^{1,\infty}(\R),
\]
where $W^{1,\infty}(\R)$ denotes the space of Lipschitz functions.

In \cite{CoEs1} Constantin and Escher proved that if $u_0\in H^1(\R)$ with $u_0-\partial_x^2u_0\in \mathcal {M}^+(\R)$, where $\mathcal{M}^+(\R)$ denotes the set of positive Randon measures with bounded total variation, then the IVP for the RCH equation \eqref{RCH} has a global weak solution $u\in L^{\infty}((0, \infty):H^1(\R))$.

 In \cite{CoMo} Constantin and Molinet improved the previous result by showing that if $u_0\in H^1(\R)$ with $u_0-\partial_x^2u_0\in \mathcal {M}^+(\R)$,
then the  IVP for the RCH equation \eqref{RCH} has a unique solution 
$$
u\in C([0,\infty):H^1(\R))\cap C^1((0,\infty):L^2(\R))
$$
 satisfying that
$ y(t)\equiv u(\cdot,t)-\partial_x^2u(\cdot,t)\in \mathcal {M}^+(\R)$ is uniformly bounded in $[0,\infty)$.

In \cite{XZ}  Xi and Zhong proved the existence of a $H^1$-global weak solution for the IVP for the RCH equation \eqref{RCH} for 
data $u_0\in H^1(\R)$.

More recently, Bressan and Constantin \cite{BC} and Bressan-Chen-Zhang \cite{BCZ}  established the existence and uniqueness, 
 respectively,  of a $H^1$ global solution for the RCH equation \eqref{RCH}. More precisely, this solution $u=u(x,t)$  is a H\"older continuous function defined in $\R \times [0,T]$ for any $T>0$ such that: 
 \begin{itemize}
\item[(i)]  for any $t\in [0,T]$,  $\,u(\cdot,t)\in H^1(\R)$,\\
\vspace{-10pt}
\item[(ii)] the map $t \to u(\cdot,t)$ is  Lipschitz continuous from $[0,T]$ to $L^2(\R)$, 
and 
\item[(iii)] it  satisfies the equation \eqref{RCH} in $L^2(\R)$ for a.e. $t\in[0,T]$.
\end{itemize}
For other well-posedness results see also \cite{GHR1}, \cite{GHR2}  and reference therein.
\vskip.1in

In the periodic case, de Lellis, Kappeler, and Topalov \cite{LKT} obtained existence, uniqueness and continuous dependence results {analogous}
 to what we shall prove here for the case of real line, in Theorem \ref{thm1}.   
Both proofs rely  on the  formulation of the IVP as an ordinary differential equation in $H^1(\R)\cap W^{1,\infty}(\R)$,
although our formulation also allows us to  examine propagation of regularity, in Theorem \ref{thm2}.

 The CH equation \eqref{CH} does not have the finite propagation speed property. In fact, if a non-trivial datum $u_0\in H^s(\R)$, $s>3/2$, has compact support, then the corresponding solution $u(\cdot,t)$ of the RCH equation \eqref{RCH} cannot have compact support any other time 
 $t\ne 0$. An
 even sharper result in this direction is given in Theorem \ref{A} of \cite{HMPZ}. However, one has that formally the RCH equation \eqref{RCH} for $u=u(x,t)$  can be rewritten in terms of
\[
m=m(x,t)=(1-\partial_x^2)u(x,t),
\]
as
\begin{equation}\label{meq}
\partial_tm+ u\, \partial_xm +2\partial_xu \,m=0, \hskip10pt \;t,\,x\in\R.
\end{equation}
Therefore, if  the data $u_0\in H^s(\R), s\geq2$, has compact support, then   $m(\cdot,t)=(1-\partial_x^2)u(\cdot,t)$, which satisfies the equation \eqref{meq}, will have compact support on the time  interval of
existence of the $H^2$-solution. This is similar to the case of the incompressible Euler equation, and the relation between the velocity and the vorticity.

Our first goal here is to establish the local well-posedness of the IVP associated to the CH equation \eqref{CH},
for a data class which includes the peakon solutions:

\begin{theorem}\label{thm1}
Given  $u_0\in X\equiv H^1(\R)\cap W^{1,\infty}(\R)$, there exist a nonincreasing function $T=T(\|u_0\|_X)>0$ and a unique solution $u=u(x,t)$ of the IVP associated to the CH equation \eqref{CH}
such that
\begin{multline}
\label{class-sol}
u\in Z_T
\equiv C([-T,T]\!:\!H^1(\R))
\cap L^{\infty}([-T,T]\!:\!W^{1,\infty}(\R))\\
\cap C^1((-T,T)\!:\!L^2(\R)),
\end{multline}
with
\[
\sup_{[-T,T]}\|u(\cdot,t)\|_X=\sup_{[-T,T]}(\|u(\cdot,t)\|_{1,2}+\|u(\cdot,t)\|_{1,\infty})\leq c\|u_0\|_X,
\]
for some universal constant $c>0$.
Moreover, given $B>0$, the map $u_0\mapsto u$, taking the data to the solution, is continuous from the ball $\{u_0\in X :\|u_0\|_X\le B\}$
into $Z_{T(B)}$.
\end{theorem}

\begin{remark}
\label{rem0}
This result shows that, for data in the space $X$, 
 the solution  of the CH equation is 
  as regular as the corresponding solution associated to the IVP for the inviscid Burgers' equation
\[
\partial_tu+ u \partial_xu=0, \hskip5pt \;t,\,x\in\R.
\]
It will be established in Theorem \ref{thm2}  hat that this still holds for \lq\lq local regularity\rq\rq.
\end{remark}

\begin{remark}\label{rem1}
 The strong notion of local well-posedness commonly used (see \cite{Ka}) does not hold in this case. 
In addition to existence and uniqueness, this notion of  LWP requires that the solution 
 satisfy the so called {\it persistence property}, namely that
if $u_0\in Y$, then $u\in C([0,T]:Y)$ and that the map taking data to the solution is locally continuous from $Y$ to $C([0,T]:Y)$. In particular, this strong version of
LWP  guarantees that the solution flow defines a dynamical system in $Y$. 
 In our case, by assuming that $u_0\in X=H^1(\R)\cap W^{1,\infty}(\R)$, we prove that the solution flow defines a dynamical system only in $H^1(\R)$.

This is necessary if one wants to have a class of solutions which includes the peakon solutions  \eqref{peakon}. To see this,  observe that if
$$
u^c_0(x)=c e^{-|x|}\in X=H^1(\R)\cap W^{1,\infty}(\R),\,\,\,\,\,\,\,\,c>0,
$$
then the corresponding solution 
$u^c(x,t)= c \,e^{|x-ct|} \in Z_T$, characterized in \eqref{class-sol}, has the property  that
\begin{equation}\label{ex1}
u^c\notin C([0,T^*]: W^{1,\infty}(\R)), \hskip15pt\text{for any }\hskip5pt T^*>0.
\end{equation}
This follows by noticing that  for any $h>0$
$$
\|  \partial_xu^c(\cdot,h)-\partial_xu^c(\cdot,0^+)\|_{\infty}>c.
$$

Similarly, for initial data 
$$
u^{c_j}_0(x)=c_j e^{-|x|}\in X=H^1(\R)\cap W^{1,\infty}(\R),\,\,\,\,\,\,j=1,2,\;\;\;\;c_1>c_2>0,
$$
one has  solutions
$u^{c_j}(x,t)=c_j e^{|x-c_jt|},\;j=1,2.
$
It is easy to check that
$$
\| u_0^{c_1}-u_0^{c_2}\|_{1,\infty}=c_1-c_2,
$$
and that for any $h>0$
\begin{multline}\label{ex2}
\|  \partial_xu^{c_1}(\cdot,h)-\partial_xu^{c_2}(\cdot,h)\|_{\infty}\\
\geq |  \partial_xu^{c_1}((c_1h)^{-},h)-\partial_xu^{c_2}((c_1h)^{-},h)|>c_1.
\end{multline}
Hence, the continuous dependence in  $W^{1,\infty}(\R)$,
that is,  the continuity of the map from $W^{1,\infty}(\R)$ into $C([-T,T]:W^{1,\infty}(\R))$, 
fails in any time interval $[0,T]$ for any $T>0$.

\end{remark}

\begin{remark}
\label{rem6}
 The proof of Therorem \ref{thm1} is based on a contraction principle argument for  a system written in Lagrangian coordinates. 
The loss of the persistence and the continuous dependence in $W^{1,\infty}(\R)$ described in \eqref{ex1} and \eqref{ex2}
is a consequence of the return to the original unknowns in Eulerian coordinates.

\end{remark}

\begin{remark}
An examination of the proof shows that if $\|u_0\|_X+|\kappa|\le B$, then
the existence time can be taken as a nonincreasing function $T(B)$
and the solution depends continuously both on the initial data $u_0$ and the parameter $\kappa$.
\end{remark}

\begin{remark}
\label{rem2}
From the continuous dependence in Theorem \ref{thm1}, one has that the solution deduced in that theorem is the limit 
in the $L^{\infty}([-T,T]: H^1(\R))$-topology of solutions 
corresponding to $u_0\in H^s(\R)$, $s>3/2$. This is consistent with the comments in \cite{Dai} and \cite{Par} concerning the realization of the peakon as a limit of smooth solutions.

\end{remark}

\begin{remark}
\label{rem5}
 As in \cite{LKT},  Theorem \ref{thm1} is valid in the spaces
\[
H^{s,p}(\R)\cap W^{1,\infty}(\R),\;\;\;\;\;s\in [1,1+1/p) \;\;\;\;p\in(1,\infty),
\]
 see the definition in \eqref{SS} and \cite{KP}. As mentioned above, these spaces also contain the peakons (see Corollary \ref{pro1}).
\end{remark}

To state our next result on propagation of regularity, we introduce the following notation. 
For $u_0\in X$, let $u\in Z_T$ be the local solution given by Theorem \ref{thm1}.
Since $u\in C(\tint;W^{1,\infty})$,  the system
\[
\dfrac{dx(s,t)}{dt}= u(x(s,t),t),\quad
x(s,0)=s,
\]
defines a one parameter family of homeomorphisms $t\mapsto x(\cdot,t)$, for $(s,t)\in\rr\times\tint$.
For any open set $\Omega_0\subset \rr$,  we define
the family of open sets
\[
\Omega_t=\{x(s,t):s\in\Omega_0\},\quad t\in\tint.
\]

\begin{theorem}\label{thm2}
 Let $u_0\in X$ and let $u\in Z_T$ be the corresponding local solution.   Let $\Omega_0\subset\rr$ be open.
With the notation above, we have:

{\rm (a)}  If  
\begin{equation}\label{hyp1a}
u_0|_{\Omega_0}\in H^{j,p}(\Omega_0),
\end{equation}
for some $p\in[2,\infty)$ and $j\in \Z$, with $j\geq 2$, then 
\begin{equation}\label{res1}
u(\cdot,t)|_{\Omega_t}\in H^{j,p}(\Omega_t),
\end{equation}
for any $t\in[-T,T]$.

\vskip.1in
{\rm (b)}   If   
\[
u_0|_{\Omega_0}\in C^{j+\theta},
\]
for  some $j\in \Z$, with $j\geq 1$, and $\theta\in(0,1)$, then 
\[
u(\cdot,t)|_{\Omega_t}\in C^{j+\theta},
\]
for any $t\in[-T,T]$.
\end{theorem}

\begin{remark}
 The result in Theorem \ref{thm2} part (a) holds for fractional values $j$ of the derivative in \eqref{hyp1a} and \eqref{res1}. However, to simplify the exposition we do not consider this case.

\end{remark}

We recall the following unique continuation and decay persistence property results obtained in \cite{HMPZ}:

\begin{theorem}[\cite{HMPZ}] \label{A}  Assume that for some $T>0$ and $s>3/2$,
\[
u\in C([0,T]: H^s(\R))
\]
is a strong solution of the IVP associated to the RCH equation \eqref{RCH}. If for some $\alpha\in(1/2,1)$, $u_0(x)=u(x,0)$ satisfies  
\begin{equation}\label{h1}
|u_0(x)|= o(e^{-x})\;\;\;\;\text{and}\;\;\;\;|\partial_xu_0(x)|=  O(e^{-\alpha x}),\;\;\;\;\text{as}\;\;\;x\uparrow \infty,
\end{equation}
and there exists $t_1\in (0,T]$ such that 
\[
|u(x,t_1)|= o(e^{-x}),\;\;\;\;\;\;\text{as}\;\;\;x\uparrow \infty,
\]
then $u\equiv 0$.
\end{theorem}

\begin{theorem}[\cite{HMPZ}] \label{B}  Assume that for some $T>0$ and $s>3/2$,
\[
u\in C([0,T]: H^s(\R))
\]
is a strong solution of the IVP associated to the RCH equation \eqref{RCH}. If 
for some $\theta\in(0,1)$,
$u_0(x)=u(x,0)$ satisfies 
\[
|u_0(x)|,\;\;\;\;|\partial_xu_0(x)|=  O(e^{-\theta x}),\;\;\;\;\text{as}\;\;\;x\uparrow \infty,
\]
then
\[
|u(x,t)|,\;\;\;\;|\partial_xu(x,t)|=  O(e^{-\theta x}),\;\;\;\;\text{as}\;\;\;x\uparrow \infty,
\]
uniformly in the time interval $[0,T]$.
\end{theorem}
 As a consequence of Theorem \ref{thm1} we shall obtain the following improvements of Theorems \ref{A} and \ref{B}:

\begin{theorem}\label{unique}
Assume that for some $T>0$, $u\in Z_T$
is a solution of the IVP associated to the RCH equation described in Theorem \ref{thm1}.

If $u_0(x)=u(x,0)$ satisfies 
\[
|u_0(x)|= o(e^{-x}),\;\;\;\;\text{and}\;\;\;\;|\partial_xu_0(x)|=  O(e^{-\alpha x})\;\;\;\;\text{as}\;\;\;x\uparrow \infty,
\]
 for some 
$\alpha\in(1/2,1)$,
and there exists $t_1\in (0,T]$ such that 
\[
|u(x,t_1)|= o(e^{-x}),\;\;\;\;\;\;\text{as}\;\;\;x\uparrow \infty,
\]
then $u\equiv 0$.

\end{theorem}

\begin{theorem}\label{decay} Assume that for some $T>0$ 
\begin{equation*}
\begin{split}
u\in \,& C([-T,T]\!:\!H^1(\R))\cap L^{\infty}([-T,T]\!:\!W^{1,\infty}(\R))\cap C^1((-T,T)\!:\!L^2(\R))
\end{split}
\end{equation*}
is a solution of the IVP associated to the RCH equation described in Theorem \ref{thm1}.

If $u_0(x)=u(x,0)$ satisfies that for some $\theta\in(0,1)$
\[
|u_0(x)|,\;\;\;\;|\partial_xu_0(x)|=  O(e^{-\theta x})\;\;\;\;\text{as}\;\;\;x\uparrow \infty,
\]
then
\[
|u(x,t)|,\;\;\;\;|\partial_xu(x,t)|=  O(e^{-\theta x})\;\;\;\;\text{as}\;\;\;x\uparrow \infty,
\]
uniformly in the time interval $[0,T]$.
\end{theorem}

\begin{remark}
\label{rem4}
Since the class of solutions considered in Theorem \ref{unique} contains the peakons  is clear that Theorem \ref{unique} is an optimal  version of Theorem \ref{A}. We observe  that with minor modifications Theorem \ref{decay} applies to solutions of the CH equation \eqref{CH}. However, we do not know whether or not
 the result in Theorem \ref{unique} can be extended to solutions of the CH equation \eqref{CH}.
\end{remark}

\begin{remark}
\label{rem7}
 As it was pointed out in \cite{Par} for the CH equation \eqref{CH} with $\k\neq 0$ the presence of the 
 {linear  dispersive term  $\kappa\partial_x(1-\partial_x^2)^{-1}u\;$} prevents the existence of non-smooth solitary waves. 
 However, Theorem \ref{thm2} shows that even in this case there is not improvement of regularity of the solution either in the $H^{s,p}$-scale or in the $C^{k+\theta}$-scale.
\end{remark}

\begin{remark}
\label{rem8}
In  \cite{ILP-cpde} Isaza, Linares, and Ponce initiated the study of the propagation of regularity for dispersive equations
considering the KdV equation \eqref{KdV}. They established the following result.
\begin{theorem}[\cite{ILP-cpde}]\label{ilp}
If  $u_0\in H^{{3/4}^{+}}(\R)$ and for some $\,l\in \Z,\,\;l\geq 1$ and $x_0\in \R$
\begin{equation*}
\|\,\partial_x^l u_0\|^2_{L^2((x_0,\infty))}=\int_{x_0}^{\infty}|\partial_x^l u_0(x)|^2dx<\infty,
\end{equation*}
then the solution $u=u(x,t)$ of the IVP associated to \eqref{KdV} provided by the local theory  in \cite{kpv-91}  satisfies  that for any $v>0$ and $\eps>0$
\begin{equation*}
\underset{0\le t\le T}{\sup}\;\int^{\infty}_{x_0+\eps -vt } (\partial_x^j u)^2(x,t)\,dx<c,
\end{equation*}
for $j=0,1, \dots, l$ with $c = c(l; \|u_0\|_{{3/4}^{+},2};\|\,\partial_x^l u_0\|_{L^2((x_0,\infty))} ; v; \eps; T)$.

In particular, for all $t\in (0,T]$, the restriction of $u(\cdot, t)$ to any interval of the form $(a, \infty)$ belongs to $H^l((a,\infty))$.

Moreover, for any $v\geq 0$, $\eps>0$ and $R>0$ 
\begin{equation*}
\int_0^T\int_{x_0+\eps -vt}^{x_0+R-vt}  (\partial_x^{l+1} u)^2(x,t)\,dx dt< c,
\end{equation*}
with  $c = c(l; \|u_0\|_{_{{3/4}^{+},2}};\|\,\partial_x^l u_0\|_{L^2((x_0,\infty))} ; v; \eps; R; T)$.
\end{theorem}


Comparing Theorem \ref{thm2} with Theorem \ref{ilp} and those in Kato \cite{Ka},
 one can conclude that solutions of the CH equation, contrary
to those of the KdV equation, do not gain regularity regardless of the smoothness and the decay of the data.

\end{remark}

\begin{remark} 
\label{rem9}
 Theorems \ref{thm2},  \ref{unique} and  \ref {decay} extend to solutions of the IVP associated to the Degasperis-Procesi (DP) equation \cite{DP}
\begin{equation}
\label{DPeq}
\partial_tu +4 u \partial_xu-\partial_t \partial_x^2u=3\partial_xu \partial_x^2u+ u\partial_x^3u, \hskip5pt \;t,\,x\in\R.
\end{equation}
In this case, the proof is simpler since the DP equation can be written as 
$$
\partial_tu+u\partial_x u= -\partial_x(1-\partial_x^2)^{-1}(3u^2/2),
$$
where the right hand side of the equations can be regarded as a ``lower order term''. This is not the case with the CH equation which can be
written as 
$$
\partial_tu+u \partial_xu= -\partial_x(1-\partial_x^2)^{-1}(\kappa u+u^2+(\partial_xu)^2/2).
$$

Thus, we have: 

\begin{theorem}\label{DG}

Under the same hypothesis,  the conclusions  in Theorems \ref{thm1},  \ref{thm2} and  \ref{unique} hold for solutions of the IVP associated to the DP equation \eqref{DPeq}.
\end{theorem}
\end{remark}

\begin{remark}
In \cite{CoEs1} and \cite{CoEs2} Constantin and Escher (see also \cite{LO}) deduced conditions on the data $u_0\in H^3(\R)$ which guarantee that the
corresponding  local solution $u\in C([0,T]:H^3(\R))$ of the IVP associated to the RCH \eqref{RCH} blows up in finite time
by showing that
$$
\lim_{t\uparrow T}\|\partial_xu(\cdot,t)\|_{\infty}=\infty,
$$
corresponding to the breaking of waves.
Observe that $H^1$-solutions of the CH equation \eqref{CH} satisfy the conservation law
$$
E(u)(t)=\int_{-\infty}^{\infty}(u^2+(\partial_xu)^2)(x,t)dx= E(u_0),
$$
so that the $H^1$-norm of the solutions constructed Theorem \ref{thm1} remains invariant 
within the existence interval.
 This highlights a sharp difference between the blow up of the CH equation and that of
  the inviscid Burgers' equation. Although in both cases the $L^{\infty}$-norm of the $x$-derivative becomes unbounded at the critical time, 
 for the CH equation the $H^1$-norm remains bounded and for Burgers' equation the $H^{1/2}$-norm
 becomes unbounded.

\end{remark}

The rest of this work is organized as follows: Section 2 contains some preliminary results to be used in the coming proofs.  The statements on existence, uniquenss, and continuous dependence given in Theorem \ref{thm1} 
will be proven in Section 3 in a series of results. Section 4 contains the proof of Theorem \ref{thm2}
on propagation of regularity, and Section 4 the proofs of Theorems \ref{unique} and  \ref{decay}. 
Since the proof of Theorem \ref{DG} is quite similar to those previously given it will be omitted.

\section{Preliminaries}


\subsection{Notation and definitions}
The standard Sobolev spaces are defined by
\begin{equation}\label{SS}
H^{s,p}(\rr)=(1-\partial_x^2)^{-1/2}L^p(\rr),\quad s\in\rr,\;1\le p<\infty,
\end{equation}
with
\[
H^s(\rr)=H^{s,2}(\rr).
\]
We define the Sobolev space 
\[
W^{1,\infty}(\rr)=\{f\in L^\infty(\rr): f'\in L^\infty(\rr)\},
\]
where the derivative is taken in the sense of distributions, and the class of Lipschitz functions 
\[
\lip=\left\{f\in L^\infty(\rr):\sup_{s_1\ne s_2}\left|\frac{f(s_1)-f(s_2)}{s_1-s_2}\right|<\infty\right\}.
\]

For notational convenience, define the functional spaces
\[
X=H^1(\rr)\cap W^{1,\infty}(\rr),
\quad 
Y=L^2(\rr)\cap L^\infty(\rr),
\quad
\xx=X\times X\times Y.
\]
The basic Lagrangian quantities and their natural spaces are:
\begin{align*}
&\xi(s,t)\in C^1(\tint : X)&&\text{displacement}\\
&x(s,t)=s+\xi(s,t)&&\text{deformation}\\
&z(s,t)\in C^1(\tint : X)&&\text{velocity}\\
&w(s,t)\in C^1(\tint : Y)&&\text{velocity gradient}
\end{align*}
and the corresponding 
Eulerian quantities are:
\begin{align*}
&s(x,t) &&\text{reference map, i.e.\ $s(x(s,t),t)=s$}\\
&S\xi(x,t)=\eta(x,t)=s(x,t)-x&& \text{reference map displacement}\\
&u(x,t)=z(s(x,t),t) &&\text{velocity}\\
&\pds{x}{u}(x,t)=w(s(x,t),t)&&\text{velocity gradient}
\end{align*}
The convolution kernel for $(1-\partial_x)^{-1}$ is denoted by
\[
G(x)=\tfrac12\exp(-|x|).
\]

\subsection{Lipschitz functions and $\ W^{1,\infty}(\rr)$}\hskip10pt

The proofs  of the following statements are not difficult and will be omitted.

\begin{lemma}
\label{ftc}
Let $f\in W^{1,\infty}(\rr)$ and define 
\[
h(s)=\int_0^sf'(\sigma)d\sigma.
\]
Then
\[
h(s)=f(s)+c\quad\text{for all $s\in\rr$},
\]
and
\[
|f(s_1)-f(s_2)|\le\norm{f'}{L^\infty}|s_1-s_2|,\quad \text{for  all}\quad s_1,s_2\in\rr.
\]
\end{lemma}

\begin{lemma}
\label{ldt}
If $f\in W^{1,\infty}(\rr)$, then
\[
\lim_{a\to0}\frac{f(s+a)-f(s)}{a}=f'(s),\quad \alev
\]
\end{lemma}

\begin{lemma}
\label{pwdiff}
If $f\in \lip$, then $f$ is differentiable almost everywhere, i.e.\
\[
\lim_{a\to0}\frac{f(s+a)-f(s)}{a}=g(s),\alev
\]
and  $g=f'$  in $\dd'$. As a consequence, if $f\in Y\cap \lip$
and $g\in Y$, then $f\in X$.
\end{lemma}

\subsection{Deformations}

\begin{lemma} 
\label{homeo}
If $\xi\in Y$, with $1+\essinf\xi'> \rho>0$,  and $x(s)=s+\xi(s)$, then
\[
\rho(s_2-s_1)\le x(s_2)-x(s_1)\le (1+ \norm{\xi'}{L^\infty})(s_2-s_1),
\]
for all $s_2>s_1$,
$x(s)$ is strictly increasing, and $x:\rr\to\rr$ is a homeomorphism.

Finally, we have 
\[
\rho\le x'(s)\le 1+ \norm{\xi'}{L^\infty}\alev
\]

\end{lemma}

\begin{lemma}
\label{homeoinv}
If  $\xi\in Y$, with $1+\essinf\xi'>\rho> 0$, then $x(s)=s+\xi(s)$
 has a strictly increasing inverse function $s:\rr\to\rr$ which satisfies
\[
(1+\norm{\xi'}{L^\infty})^{-1}(x_2-x_1)\le s(x_2)-s(x_1)\le \rho^{-1}(x_2-x_1),
\]
for all $ x_2>x_1$, 
\[
s'(x)=(x'\circ s(x))^{-1}\alev,
\]
and
\[
(1+\norm{\xi'}{L^\infty})^{-1}\le s'(x)\le \rho^{-1}\alev
\]

\end{lemma}

\begin{lemma}\label{ellone}
Let  $\xi\in Y$, with $1+\essinf\xi'> \rho>0$,  and define $x(s)=s+\xi(s)$.    
 If $f\in L^1(\rr)$, 
then $f\circ x\in L^1(\rr)$,
\[
\intr f\circ x(\sigma)x'(\sigma)d\sigma =\intr f(x)dx,
\]
and $\norm{f\circ x}{L^1}\le \rho^{-1}\norm{f}{L^1}$.

Similarly, if $s(x)$ is the inverse function, then $f\circ s\in L^1(\rr)$,
\[
\intr f\circ s(x) s'(x)dx = \intr f(s) ds,
\]
and $\norm{f\circ s}{L^1}\le (1+\norm{\xi'}{L^\infty})\norm{f}{L^1}$.
\end{lemma}

\begin{corollary}\label{ellpee}
Let $\xi\in Y$, with $1+\essinf\xi'>\rho>0$,  and define $x(s)=s+\xi(s)$.    
 If $f\in L^p$, $1\le p \le\infty$,
then $f\circ x\in L^p$ and 
\[
\norm{f\circ x}{L^p}\le \rho^{-1/p}\norm{f}{L^p}.
\]
Similarly, if $s(x)$ is the inverse function, then $f\circ s\in L^p$
and 
\[
\norm{f\circ s}{L^p}\le (1+\norm{\xi'}{L^\infty})^{1/p}\norm{f}{L^p}.
\]
\end{corollary}

\begin{definition}\label{tdef}
We define the displacement domain
\[
\oo=\{\xi\in X: 1+\essinf \xi'>\rho\}, \quad 0<\rho\ll1.
\]
Given $\xi\in\oo$,  set $x(s)=s+\xi(s)$, and let $s(x)$ be the inverse function described in Lemma \ref{homeoinv}.
Define the mapping $S\xi(x)=s(x)-x$.
\end{definition}

\begin{lemma}\label{homeo3}
The mapping $S$  in Definition \ref{tdef} satisfies
\[
S:\oo\to X
\]
and
\[
\norm{S\xi}{X}\le C(\rho,\norm{\xi'}{L^\infty})\norm{\xi}{X}.
\]
\end{lemma}
   
\begin{proof}
Let $\eta=S\xi$.
  Then by Lemmas \ref{homeoinv} and \ref{ellone} we have the following:
\[
\eta(x)=-\xi\circ s(x),
\]
\[
\norm{\eta}{L^\infty}=\norm{\xi}{L^\infty},\quad \norm {\eta'}{L^\infty}\le \rho^{-1}\norm{\xi'}{L^\infty},
\]
and
\[
\norm{\eta}{L^2}\le(1+\norm{\xi'}{L^\infty})\norm{\xi}{L^2}, 
\quad
\norm{\eta'}{L^2}\le \rho^{-1/2}\norm{\xi'}{L^2}.
\]

\end{proof}

\begin{lemma}\label{ycont}
If $\xi_j\in \oo$, $j=1,2$, then
\[
\norm{S\xi_1- S\xi_2}{L^\infty}\le \rho^{-1}\norm{\xi_1-\xi_2}{L^\infty},
\]
and
\[
\norm{S\xi_1- S\xi_2}{L^2}\le C(1+\norm{\xi'_1}{L^\infty}+\norm{\xi'_2}{L^\infty})^{1/2}\norm{\xi_1-\xi_2}{L^2}.
\]
\end{lemma}

\begin{proof}
Set $\eta_j=S\xi_j$ and $s_j(x)=x+\eta_j(x)$, $j=1,2$.  

Fix $x$, and assume that $s_1(x)>s_2(x)$.  Then by Lemmas \ref{homeo} and \ref{homeo3}
\allowdisplaybreaks
\begin{equation*}
\begin{split}
\rho\;(\eta_1(x)-\eta_2(x)&)=\rho\;(s_1(x)-s_2(x))\\
\le& \frac12(x_1\circ s_1(x)-x_1\circ s_2(x))+\frac12(x_2\circ s_1(x)-x_2\circ s_2(x))\\
=&\frac12(\xi_2\circ s_2(x)-\xi_1\circ s_2(x)-\xi_1\circ s_1(x)+\xi_2\circ s_1(x)).
\end{split}
\end{equation*}
If $s_1(x)<s_2(x)$, then a similar inequality holds with the subscripts 1 and 2 interchanged.
Therefore, we have that
\begin{equation*}
\rho\;|\eta_1(x)-\eta_2(x)|
\le \frac12|\xi_1\circ s_1(x)-\xi_2\circ s_1(x)|+\frac12|\xi_1\circ s_2(x)-\xi_2\circ s_2(x)|.
\end{equation*}

From this it follows that
\[
\rho\;\norm{S\xi_1-S\xi_2}{L^\infty}
=\rho\;\norm{\eta_1-\eta_2}{L^\infty}
\le\norm{\xi_1-\xi_2}{L^\infty}.
\]
 Also by   Corollary \ref{ellpee} we have that
\begin{equation*}
\begin{split}
\rho\;\|S\xi_1-&S\xi_2\|_{L^2}
=\rho\;\norm{\eta_1-\eta_2}{L^2}\\
\le&\norm{\xi_1\circ s_1-\xi_2\circ s_1}{L^2}
+\norm{\xi_1\circ s_2-\xi_2\circ s_2}{L^2}\\
\le&(1+\norm{\xi'_1}{L^\infty})^{1/2}\norm{\xi_1-\xi_2}{L^2}+(1+\norm{\xi'_2}{L^\infty})^{1/2}\norm{\xi_1-\xi_2}{L^2}\\
\lesssim&(1+\norm{\xi'_1}{L^\infty}+\norm{\xi'_2}{L^\infty})^{1/2}\norm{\xi_1-\xi_2}{L^2}.
\end{split}
\end{equation*}
\end{proof}

\begin{lemma}\label{homeo4}
Let $f\in L^2(\rr)$.  With the notation of Definition \ref{tdef}, the mapping
\[
\xi\mapsto f\circ s
\]
is locally uniformly continuous from  $\oo$ into $L^2(\rr)$.
\end{lemma}

\begin{proof}
Choose any $B>0$, and define the bounded set
\[
N=\{\xi\in\oo: \norm{\xi}{X}\le B\}.
\]
Let $\xi_j\in N$, and set
\[
x_j(s)=s+\xi_j(s)\quad \text{and}\quad s_j(x)=x+S\xi_j(x),\quad j=1,2.
\]
If $\phi\in C_0^\infty(\rr)$,  then by Lemma \ref{ellone}, the mean value theorem,
and Lemma \ref{ycont}, we have 
\begin{equation*}
\begin{split}
\|f\circ s_1-f\circ s_2\|_{L^2}
\le &\, C(B)\;\norm{f -\phi}{L^2}+C(\rho,B)\;\norm{\phi'}{L^\infty}\;\norm {\xi_1-\xi_2}{L^2}.
\end{split}
\end{equation*}
Let $\eps >0$ be given.
Since $C_0^\infty(\rr)$ is dense in $L^2(\rr)$, we can choose $\phi$ depending only on $f$ and $B$ so that the first term is smaller than $\eps/2$.
So if
\[
\norm {\xi_1-\xi_2}{L^2}<\delta,
\]
then by choosing $\delta$ sufficiently small, the second term is also smaller than $\eps/2$.
This proves uniform continuity on $N$.
\end{proof}


\begin{lemma}\label{homeo5}
The mapping $D_xS:\oo\to L^2(\rr)$ is   continuous.
\end{lemma}  

 \begin{proof}
Let $\xi_j\in N$, as  in the proof of Lemma \ref{homeo4}, and set
\[
x_j(s)=s+\xi_j(s)\quad \text{and}\quad s_j(x)=x+S\xi_j(x),\quad j=1,2.
\]
Then by Lemma \ref{homeoinv},
\begin{equation*}
D_xS\xi_j(x)= s_j'(x)-1
=\frac1{x'_j\circ s_j(x)}-1
=-\frac{\xi'_j\circ s_j(x)}{x'_j\circ s_j(x)}.
\end{equation*}

Therefore, by Lemma \ref{homeo}  and the fact that $x_j'(s)=1+\xi_j'(s)$, we see that
\begin{align*}
|D_xS\xi_1(x)-&D_xS\xi_2(x)|\\
=&\left|\frac{-\xi'_1\circ s_1(x)\;x_2'\circ s_2(x)+\xi'_2\circ s_2(x)\;x'_1\circ s_1(x)}{x'_1\circ s_1(x)\;x'_2\circ s_2(x)}\right|\\
=&\left|\frac{-\xi'_1\circ s_1(x)+\xi'_2\circ s_2(x)}{x'_1\circ s_1(x)\;x'_2\circ s_2(x)}\right|\\
\le&\rho^{-2}|\xi'_1\circ s_1(x)-\xi'_2\circ s_2(x)|.
\end{align*}

From this and the triangle inequality we get
\begin{equation*}
\norm{D_xS\xi_1-D_xS\xi_2}{L^2}\le  \rho^{-2}(\norm{\xi'_1\circ s_1-\xi'_1\circ s_2}{L^2}+\norm{\xi'_1\circ s_2-\xi'_2\circ s_2}{L^2}).
\end{equation*}

By Corollary \ref{ellpee}, the second term is estimated by
\[
\norm{\xi'_1\circ s_2-\xi'_2\circ s_2}{L^2}\le \rho^{-1/2} \norm{\xi'_1-\xi'_2}{L^2}.
\]
Since $\xi_1'\in L^2$,   continuity at $\xi_1$ now follows  by  Lemma \ref{homeo4}.
\end{proof}


\begin{lemma}
\label{cov}
Let $\xi\in\oo$,  define $x(s)=s+\xi(s)$.  If $M\in Y$, then
\[
\int_{-\infty}^{s}\exp(x(\sigma))M(\sigma) x'(\sigma)d\sigma=\int_{-\infty}^{x(s)}\exp(y)M(s(y))dy
\]
and
\[
\int_s^{\infty}\exp(-x(\sigma))M(\sigma) x'(\sigma)d\sigma=\int_{x(s)}^{\infty}\exp(-y)M(s(y))dy.
\]
\end{lemma}

\begin{proof}
Define
\[
h(x)=\int_{-\infty}^x\exp(y)M(s(y))dy.
\]
By the Lebesgue Differentiation Theorem,
\[
h'(x)=\lim_{a\to0}\frac{h(x+a)-h(x)}{a}=\exp(x)M(s(x)),\alev
\]
By the chain rule,
\[
\od{}{s}h(x(s))=\exp(x(s))M(s)x'(s),\alev
\]
By the same idea as Lemma \ref{ftc}
\[
h(x(s))=\int_{-\infty}^{s}\exp(x(\sigma))M(\sigma) x'(\sigma)d\sigma.
\]

\end{proof}


\subsection{Properties of the kernel}

\begin{lemma}
\label{conv}

If $\xi\in\oo$ and $x(s)=s+\xi(s)$, then
\[
G'\xsmxs=-G\xsmxs\sgn(s-\sigma),\quad s\ne\sigma,
\]
and
\[
G\xsmxs\le G(\rho\;(s-\sigma)).
\]
\end{lemma}

\begin{proof} The results follow by definition and Lemma \ref{homeo}.
\end{proof}

\begin{lemma}
\label{glip}
Let $\xi_j\in \oo$, $j=1,2$, and set $x_j(s)=s+\xi_j(s)$.
Then
\begin{multline*}
|G(x_1(s)-x_1(\sigma))-G(x_2(s)-x_2(\sigma))|\\
\le 
G(\rho(s-\sigma))
\cdot(|\xi_1(s)-\xi_2(s)|+|\xi_1(\sigma)-\xi_2(\sigma)|)
\end{multline*}
and
\begin{multline*}
|G'(x_1(s)-x_1(\sigma))-G'(x_2(s)-x_2(\sigma))|\\
\le 
G(\rho(s-\sigma))
\cdot(|\xi_1(s)-\xi_2(s)|+|\xi_1(\sigma)-\xi_2(\sigma)|),
\end{multline*}
for all $s,\sigma\in\rr$.
\end{lemma}

\begin{proof}
Assume, without loss of generality, that $s>\sigma$.  
By Lemma \ref{homeo}, 
\begin{equation}
\label{xslb}
x_j(s)-x_j(\sigma)\ge\rho(s-\sigma)>0.
\end{equation}
Then by the mean value theorem, we have
\begin{align*}
G(x_1(s)&-x_1(\sigma))-G(x_2(s)-x_2(\sigma))\\
&=\frac12\exp A\cdot [(-\xi_1(s)+\xi_2(s))+(\xi_1(\sigma)-\xi_2(\sigma))],
\end{align*}
where $A$ lies between
\[
-(x_1(s)-x_1(\sigma))
\quad\text{and}\quad
-(x_2(s)-x_2(\sigma)).
\]
Therefore, by \eqref{xslb}, we get $\frac12\exp A\le G(\rho(s-\sigma))$, and the first inequality follows.
The second inequality follows from the first and Lemma \ref{conv}.
\end{proof}

\vspace{.5mm}

We will now briefly discuss the claim made in   Remark \ref{rem5}.

\begin{proposition}\label{pro1b} If $f\in L^p(\R)$, $p\in(1,\infty)$, and there exists $x_0\in \R$ such that
$f(x_0^+)$ and $f(x_0^-)$ are defined  and $f(x_0^+)\neq f(x_0^-)$, then $f\notin H^{1/p,p}(\R)$.
\end{proposition}

\begin{proof}
This is a direct consequence of the following characterization of the spaces $H^{s,p}(\R^n)$  
established in \cite{Stein}: 
Given $s\in(0,1)$, $p\in(1,\infty)$, and $f\in L^p(\R^n)$, then $f\in H^{s,p}(\R^n)$ 
if and only if
\[
\mathcal D^{s}f(x)=\lim_{{\eps\downarrow 0}}\int_{|y|\geq \eps}\frac{f(x+y)-f(x)}{|y|^{n+{s}}}dy
\]
is defined in $L^p(\R^n)$. In this case
$$
\|f\|_{s,p}=\|(1-\partial_x^2)^{s/2}f\|_p\sim \|f\|_p+\| \mathcal D^s f\|_p.
$$

\end{proof}
\begin{corollary} \label{pro1} For any $p\in (1,\infty)$
\[
\exp(-|x|)\in H^{s,p}(\R),\;\;\;\;\;\;\;\;s\in(0,1+1/p)\,
\]
but 
\[
\exp(-|x|)\notin H^{1+1/p,p}(\R).
\]
\end{corollary}

For further details see Chapter 3 in \cite{LP}.

\vskip.5mm

We end this section proving a result useful in the remainder of this paper.

\begin{lemma} \label{pro2}
Let $f\in L^1(\R)$ and  $a, b\in \R$. Recall that $G(x)= \frac12\exp(-|x|)$.
\vspace{3mm}

{\rm(a)} If for some  $j\in \Z^+$ and $p\in[1,\infty)$
\begin{equation}
\label{pre2}
f|_{(a,b)}\in H^{j,p}(a,b),
\end{equation}
then for any $\eps>0$
\begin{equation}
\label{pre3}
G\ast f|_{(a,b)}\in H^{j+1,p}((a+\eps,b-\eps)),
\end{equation}

and 
\vskip.1in
{\rm(b)} if for some $j\in \Z^+$ and $\theta\in (0,1]$
\begin{equation}
\label{pre2a}
f|_{(a,b)}\in C^{j,\theta}(a,b),
\end{equation}
then 
\begin{equation}
\label{pre3a}
G\ast f|_{(a,b)}\in C^{j+1,\theta}(a,b).
\end{equation}
\end{lemma}

\begin{proof}
For any $\eps>0$ let $\varphi_{\eps}\in C^{\infty}_0(\R)$ with $\varphi_{\eps}(x)\geq 0$,
$$
\varphi_{\eps}(x)=1,\;x\in(a+2\eps/3,b-2\eps/3),\;\; \supp (\varphi_{\eps})\subset (a+\eps/3,b-\eps/3).
$$
Define 
$$
v(x)=G\ast f(x)=G\ast (f\varphi_{\eps})(x)+ G\ast (f(1-\varphi_{\eps}))(x)=v_1(x)+v_2(x).
$$
Since, $\partial_xG\in L^1(\R)$ it is easy to see that  assuming  \eqref{pre2} (resp. \eqref{pre2a}), $v_1$ satisfies \eqref{pre3} (resp. \eqref{pre3a}).

 By observing that $v_2\in C^{\infty}(a+\eps,b-\eps)$ one obtains the desired result.
\end{proof}
It is clear that by using Young's inequality  the result in Lemma \ref{pre2} part (a) extends to the case where $a=-\infty\,$ or $\,b=\infty$.

The result in Lemma \ref{pro2} extends to fractional values of $j$ in \eqref{pre2} and \eqref{pre3}. However, to simplify the exposition we restrict ourselves to $j\in \Z^+$.

\vspace{.5mm}

\section{The Initial Value Problem}

Here we will establish the local well-posedness for the IVP associated to the CH equation, that is,
\begin{equation}\label{CHb}
\begin{cases}
\partial_tu+\k \partial_xu +3 u \partial_xu-\partial_t \partial_x^2 u=2\partial_xu \partial_x^2u+ u\partial_x^3u, \hskip5pt \;t,\,x,\,\kappa\in\R,\\
u(x,0)=u_0(x).
\end{cases}
\end{equation}
The proof of Theorem \ref{thm1} will be given in several stages.

We first prove some estimates for the nonlinear terms appearing in the equation in \eqref{CHb}. 
For this aim we will use the notation and
 estimates from the previous section. 
To simplify the presentation, define the
nonlinear functions
\begin{align*}
&M(z,w)=2\kappa z+z^2+\frac12w^2,\quad \kappa \in\rr,\\
&N(z,w)=M(z,w)-w^2.
\end{align*}
Recall the definition of $\oo$ in Definition \ref{tdef}.
Given
\[
\xzw\in\oxy,
\]
we shall consider the
deformations of the form
\[
x(s)=s+\xi(s),
\]
and we define the nonlinear mappings
\begin{align*}
&F_1\xzw(s)=-\intr G'(x(s)-x(\sigma))M(z,w)(\sigma)x'(\sigma)d\sigma\\
&F_2\xzw(s)=-\intr G(x(s)-x(\sigma))M(z,w)(\sigma)x'(\sigma)d\sigma+N(z,w)(s).
\end{align*}

\subsection{Estimates of nonlinear mappings}

\begin{lemma}
\label{mnest}
If $(z,w)\in\ Y\times Y$, then
\begin{align*}
&\norm{M(z,w)}{L^\infty}+\norm{N(z,w)}{L^\infty}\lesssim \norm{z}{L^\infty}+\norm{z}{L^\infty}^2+\norm{w}{L^\infty}^2,\\
&\norm{M(z,w)}{L^2}+\norm{N(z,w)}{L^2}\lesssim \norm{z}{L^2}+\norm{z}{L^2}\norm{z}{L^\infty}+\norm{w}{L^2}\norm{w}{L^\infty},\\
\intertext{and}
&\norm{M(z,w)}{Y}+\norm{N(z,w)}{Y}\lesssim \norm{z}{Y}+\norm{z}{Y}^2+\norm{w}{Y}^2.
\end{align*}
\end{lemma}

\begin{proof}
This follows directly from the definitions of $M$, $N$, and $Y$.
\end{proof}

\begin{lemma}
\label{mnlip}
If $(z_j,w_j)\in Y\times Y$, $j=1,2$, then
\begin{multline*}
\norm{M(z_1,w_1)-M(z_2,w_2)}{Y}+\norm{N(z_1,w_1)-N(z_2,w_2)}{Y}\\
\lesssim
(1+\norm{z_1}{Y}+\norm{z_2}{Y})\norm{z_1-z_2}{Y}\\
+(\norm{w_1}{Y}+\norm{w_2}{Y})\norm{w_1-w_2}{Y}.
\end{multline*}
\end{lemma}

\begin{proof}
Similar to Lemma \ref{mnest}.
\end{proof}

\begin{lemma}
\label{fmap}
If $\xzw\in \oxy$, then
\begin{multline*}
|F_1\xzw(s)|+|F_2\xzw(s)|\\
\lesssim\intr G(\rho(s-\sigma))|M(z,w)(\sigma)|d\sigma\;\norm{x'}{L^\infty}
+|N(z,w)(s)|,
\end{multline*}
\begin{multline*}
\norm{F_1\xzw}{L^\infty}+\norm{F_2\xzw}{L^\infty}\\
\lesssim\rho^{-1}\norm{G}{L^1}\norm{M(z,w)}{L^\infty}\norm{x'}{L^\infty}+\norm{N(z,w)}{L^\infty},
\end{multline*}
and
\begin{multline*}
\norm{F_1\xzw}{L^2}+\norm{F_2\xzw}{L^2}\\
\lesssim\rho^{-1/2}\norm{G}{L^2}\norm{M(z,w)}{L^2}\norm{x'}{L^\infty}+\norm{N(z,w)}{L^2}.
\end{multline*}
Moreover, $F_j:\oxy\to Y$, $j=1,2$.
\end{lemma}

\begin{proof}
The first inequality follows from the definitions of $F_j$, $j=1,2$, and Lemma \ref{conv}.

The other two inequalities follow from the first one using Young's inequality.

The final statement follows from these and Lemma \ref{mnest}.
\end{proof}

\begin{lemma}
\label{harddiff}
Let $\xzw\in\oxy$.  Then
\[
\od{}{s}F_1\xzw(s)=\paren{F_2\xzw(s)+w^2(s)}x'(s), \text {$\alev$ and in $\dd'$}.
\]
The map
\[
\xzw\mapsto F_1\xzw
\]
takes $\oxy$ into $X$.
\end{lemma}

\begin{proof}
By definition of $F_1$ and Lemma \ref{conv}, we can write
\[
F_1\xzw (s)=F_{11}(s)+F_{12}(s),
\]
where
\begin{align*}
&F_{11}(s)=\frac12\int_{-\infty}^s\exp(-x(s)+x(\sigma))M(z,w)(\sigma)x'(\sigma)d\sigma,\\
&F_{12}(s)=-\frac12\int_s^{\infty}\exp(x(s)-x(\sigma))M(z,w)(\sigma)x'(\sigma)d\sigma,
\end{align*}
and  $x(s)=s+\xi(s)$, as usual.

Since
\[
\exp x(\sigma) \le \exp\sigma\exp\norm{\xi}{L^\infty}\in L^1((-\infty,s]),
\]
we may write
\[
F_{11}(s)=\frac12\exp(-x(s))\int_{-\infty}^{s}\exp x(\sigma)M(z,w)(\sigma)x'(\sigma)d\sigma.
\]

By Lemmas \ref{ftc} and \ref{mnest}, $F_{11}\in Y\cap\lip$.
By  the chain rule, we have
\[
\od{}{s}\exp(-x(s))=-\exp(-x(s))x'(s),\alev,
\]
and by the Lebesgue Differentiation Theorem, we have
\begin{align*}
\od{}{s}\int_{-\infty}^{s}\exp x(\sigma)&M(z,w)(\sigma)x'(\sigma)d\sigma\\
&=\lim_{a\to0}\frac1a\int_{s}^{s+a}\exp x(\sigma)M(z,w)(\sigma)x'(\sigma)d\sigma\\
&=\exp x(s)M(z,w)(s)x'(s),\alev
\end{align*}
Thus, by the product rule, we get
\[
F'_{11}(s)
=( -F_{11}(s)
+(1/2)M(z,w)(s))x'(s),\alev
\]
Similarly, we have $F_{12}\in Y\cap\lip$, and
\[
F'_{12}(s)
=(F_{12}(s)
+(1/2)M(z,w)(s))x'(s),\alev
\]
Therefore, $F_1\in Y\cap\lip$, and
\begin{align*}
\od{}{s}F_1\xzw(s)
&=(-F_{11}(s)+F_{12}(s)+M(z,w)(s))x'(s)\\
&=(F_2\xzw(s)+w^2(s))x'(s),\alev
\end{align*}
Since $F_{1}'\in Y$, by Lemma \ref{mnest}, we obtain from Lemma \ref{pwdiff} that $F_1'$
is the derivative in the distributional sense and $F_{1}\in X$.
\end{proof}

\begin{lemma}
\label{effdiff}
Let $\xi_j\in \oo$,   and set $x_j(s)=s+\xi_j(s)$, $j=1,2$.  Then for $k=1,2$, we have
\begin{align*}
|F_k(\xi_1,&z,w)(s)-F_k(\xi_2,z,w)(s)|\\
\lesssim&\; |\xi_1(s)-\xi_2(s)|\intr G(s-\sigma)\;|M(z,w)(\sigma)|\;x'_1(\sigma)d\sigma\\
&+\intr G(\rho(s-\sigma))\;|\xi_1(\sigma)-\xi_2(\sigma)|\;|M(z,w)(\sigma)|\;x'_1(\sigma)d\sigma\\
&+\intr G(\rho(s-\sigma))\;|M(z,w)(\sigma)|\;|x'_1(\sigma)-x'_2(\sigma)|\;d\sigma,
\end{align*}
and
\begin{align*}
&\left| \od{}{s}[F_1(x_1,z,w)(s)-F_1(x_2,z,w)(s)]\right|\\
&\phantom{\pd{}{s}[F_1(x_1,z,w)(s)}\lesssim |F_2(\xi_1,z,w)(s)-F_2(\xi_2,z,w)(s)|\;x_1'(s)\\
&\phantom{\pd{}{s}[F_1(x_1,z,w)(s)}\qquad+(|F_2(\xi_2,z,w)(s)|+w^2(s))\;|\xi'_1(s)-\xi'_2(s)|.
\end{align*}

\end{lemma}

\begin{proof}
This follows by Lemmas \ref{glip} and \ref{harddiff}.
\end{proof}

%
%
%
%
%
%

\begin{definition}
Define the space 
\begin{equation*}
\xx=\{\vv=
(\xi,z,w)
\in X\times X\times Y\},
\end{equation*}
with the norm
\[
\norm{\vv}{\xx}=\norm{\xi}{X}+\norm{z}{X}+\norm{w}{Y}.
\]
\end{definition}

\begin{theorem}\label{effloclip}
Define the mapping
\[
\fff\xzw=
(z,  F_1\xzw,  F_2\xzw).
\]
Then $\fff:\oxy\to \xx$ is locally Lipschitz.
\end{theorem}

\begin{proof} 
Lemmas \ref{mnlip} and \ref{effdiff} yield the result.
\end{proof}

\subsection{Local existence} 

We first construct a solution in Lagrangian coordinates.

\begin{theorem}\label{localexistence}
Given $B>0$, define
\[
\neigh(\rho,B)=\{\vv_0=(\xi_0,z_0,w_0)\in\xx: \xi_0\in\oo,\;\norm{\vv_0}{\xx}<B\}.
\]
There exists a time $T>0$,
depending only upon $\rho$, $B$,  and $B-\norm{\vv_0}{\xx}$, such that the system
\[
\vv(\cdot,t)=\vv_0(\cdot)+\int_0^t\fff(\vv(\cdot,\tau))d\tau,
\]
has a unique solution $\vv\in C^1(\tint : \neigh(\rho, B))$.

If $z'_0(s)=w_0(s)(1+\xi_0'(s))$ a.e.,  then 
\[
\vv(\cdot,t)=(
\xi(\cdot,t), z(\cdot,t), w(\cdot,t)
)
\]
satisfies
\begin{equation}\label{wdxu}
\pds{s}{z}(\cdot,t)=w(\cdot,t)\paren{1+\pds{s}{\xi}(\cdot,t)},\quad \text{in $C(\tint : Y)$.}
\end{equation}
\end{theorem}

\begin{proof}
Since $\xi_0\in\oo$, we can find $\bar\rho$ such that
\[
\rho<\bar\rho<1+\essinf\xi'_0.
\]
Define the  set
\[
\bar \neigh=\{\vv=(\xi,z,w)\in\xx: 1+\essinf \xi'\ge\bar\rho,\;\norm{\vv-\vv_0}{\xx}\le B-\norm{\vv_0}{\xx}\}.
\]
Then  $\bar \neigh\ne\emptyset$ since  $\vv_0\in\bar \neigh$, $\bar \neigh$ is closed in $\xx$, and $\bar \neigh\subset \neigh(\rho, B)$.
  Apply the contraction mapping principle to the operator
  \[
  \mathbf S\vv(\cdot,t)=\vv_0(\cdot)+\int_0^t\fff(\vv(\cdot,\tau))d\tau,
\]
on the set $ C(\tint,\bar \neigh)$, with $T$ sufficiently small.

The last statement follows from
\[
\pds{t}{J}(s,t)=-w(s,t)J(s,t),\quad J(s,0)=0,
\]
where 
\[
J(s,t)=\pds{s}{z}(s,t)-w(s,t)\pds{s}{x}(s,t).
\]
\end{proof}

\subsection{Solution of the Camassa-Holm  equation \eqref{CHb}} 

Next, we establish the regularity of the local solution in Eulerian coordinates.

\begin{theorem}
\label{regthm}
Let $u_{0}\in X$, and define
\[
\vv_{0}=(0,u_0,u'_0)
\in\oo\times X\times Y.
\]
Choose any $B>\norm{\vv_0}{\mathcal X}$, and let 
\[
\vv=(\xi,z,w)
\in C^1(\tint :  \neigh(\rho, B))
\]
 be the corresponding solution from Theorem  \ref{localexistence}.
Let $x(s,t)=s+\xi(s,t)$ and let $s(x,t)=x+S\xi(x,t)$ be the inverse function.
 Define
 \[
 u(x,t)=z(s(x,t),t).
 \]
Then
\[
S\xi, u\in C(\tint: H^1)\cap C^1(\tint:L^2),
\]
\[
 \partial_xS\xi,\partial_xu\in L^\infty(\tint:Y),
\]
\begin{gather}
\label{eikonal}
\pds{t}s+u\pds{x}s=0,\quad \text{on}\quad \rr\times\tint,\\
\nonumber
s(x,0)=x,\quad x\in\rr,
\end{gather}
and 
\begin{gather}
\label{CH-char}
\partial _t u+u\partial_xu+G'\ast M(u,\partial_x{u})=0, \quad \text{on}\quad \rr\times\tint,\\
\nonumber
u(x,0)=u_0(x),\quad x\in \rr.
\end{gather}

\end{theorem}

\begin{proof}
We shall prove the following statements sequentially:
\begin{gather}
\label{elltwocont}
S\xi,u\in C(\tint: L^2),\\
\label{ellinfbd}
\pds{x}S\xi,\pds{x}u\in L^\infty (\tint:Y),\\
\label{elltwodercont}
\pds{x}S\xi,\pds{x}u\in C(\tint: L^2),\\
\label{essreg}
S\xi\in C^1(\tint: L^2)\quad \text{and \eqref{eikonal} holds, }\\
\label{youreg}
u\in C^1(\tint: L^2)\quad \text{and \eqref{CH-char} holds.}
\end{gather}

Since $\vv\in C^1(\tint:\neigh(\rho, B))$, we have that
\[
\sup_{\tint}\sum_{k=0,1}(\norm{\partial_t^k\xi(\cdot,t)}{X}+\norm{\partial_t^kz(\cdot,t)}{X}+\norm{\partial_t^kw(\cdot,t)}{Y})\le B.
\]
Throughout the proof, generic constants may  depend on $B$ and $\rho$.

Let $t,t_1\in\tint$.
By Lemma \ref{ycont}, we have that
\[
\norm{S\xi(\cdot,t_1)-S\xi(\cdot,t)}{L^2}\le C \norm{\xi(\cdot,t_1)-\xi(\cdot,t)}{L^2}.
\]
This proves \eqref{elltwocont} for $S\xi$.

Using the definition of $u$ and the fact that $z\in C^1(\tint:L^2)$, we have
\begin{align*}
u(x,t_1)-u(x,t)
=&\int_t^{t_1}\partial_t z(s(x,t_1),\tau)d\tau
+z(s(x,t_1),t)-z(s(x,t),t).
\end{align*}
Using Lemma \ref{ftc}, Corollary \ref{ellpee}, and Lemma \ref{ycont}, we obtain
\begin{equation*}
\begin{split}
&\norm{u(\cdot,t_1)-u(\cdot,t)}{L^2}\\
\le & \int_t^{t_1}\norm{\partial_tz(s(\cdot,t_1),\tau)}{L^2}d\tau+\norm{\partial_sz(\cdot,t)}{L^\infty}\norm{s(\cdot,t_1)-s(\cdot,t)}{L^2}\\
\le & \int_t^{t_1}(1+\norm{\xi'(\cdot,t_1)}{L^\infty})^{1/2}\norm{\partial_tz(\cdot,\tau)}{L^2}d\tau
 + C\norm{S\xi(\cdot,t_1)-S\xi(\cdot,t)}{L^2}\\
\le &\, C(|t_1-t|+\norm{\xi(\cdot,t_1)-\xi(\cdot,t)}{L^2}).
\end{split}
\end{equation*}
So \eqref{elltwocont} now follows for $u$.

By Lemma \ref{homeo3}, we have
\[
\norm{\pds{x}S\xi(\cdot,t)}{Y}\le\norm{S\xi(\cdot,t)}{X}\le C,
\]
which proves \eqref{ellinfbd} for $\pds{x}S\xi$.

Since $z(\cdot,t), s(\cdot,t)\in W^{1,\infty}$, we  use Lemma \ref{ldt}, the chain rule, Lemma \ref{homeoinv}, and  \eqref{wdxu}
to obtain
\begin{equation}\label{wid}
\begin{split}
\partial_xu(x,t)
&=\partial_sz(s(x,t),t)(\partial_sx(s(x,t),t))^{-1}\\
&=w(s(x,t),t)\alev
\end{split}
\end{equation}
Now by Corollary \ref{ellpee}, we have
\[
\norm{\partial_x u(\cdot,t)}{Y}= \norm{w(s(\cdot,t),t)}{Y}
\le C\norm{w(\cdot,t)}{Y}\le C.
\]
This verifies \eqref{ellinfbd} for $\pds{x}u$.

Since $\xi\in C(\tint: \oo)$ and $D_xS:\oo\to L^2$ is continuous by Lemma \ref{homeo5}, we see that
$\pds{x}S\xi\in C(\tint: L^2)$.  Since $w\in C^1(\tint: L^2)$, we also obtain that
\[
\pds{x}u(x,t)=w(s(x,t),t)\in C(\tint: L^2),
\]
 exactly as was shown above for $u(x,t)=z(s(x,t),t)$.
This establishes \eqref{elltwodercont}.


Next, we prove \eqref{essreg}.

Generally speaking, given a function $f(x,t)$ on $\rr\times \tint$ and $h\ne0$, we define
\[
R_xf(x,t;h)=\frac{f(x+h,t)-f(x,t)}{h}-\partial_xf(x,t)
\]
and
\[
R_tf(x,t;h)=\frac{f(x,t+h)-f(x,t)}{h}-\partial_tf(x,t).
\]
We define $R_xf(x,t;0)=R_tf(x,t;0)=0$.

Since $\vv\in C^1(\tint: \oo\times X\times Y)$, we have in particular that 
$\xi\in C^1(\tint: Y)$ and $z\in C^1(\tint: L^2)$.
Therefore, 
\begin{equation}
\label{xistrdiff}
\norm{R_t\xi(\cdot,t;h)}{L^2}+\norm{R_t\xi(\cdot,t;h)}{L^\infty}\to0,\;\text{as}\quad h\to0,
\end{equation}
and
\[
\norm{R_tz(\cdot,t;h)}{L^2}\to0,\quad \text{as}\quad h\to0.
\]Since the derivative of an $H^1$ function exists strongly in $L^2$, we have that
\begin{equation}
\label{elltwospder}
\norm{R_sz(\cdot,t;h)}{L^2}\to0\; \text{and}\; \norm{R_xS\xi(\cdot,t;h)}{L^2}\to0,\; \text{as}\; h\to0.
\end{equation}

Having shown that $u\in C(\tint: H^1)$, we have that $u\in C(\tint: L^\infty)$.
Since we also have that $\pds{x}S\xi \in C(\tint: L^2)$, it follows that
\[
u\pds{x}s=u(1+\pds{x}S\xi)\in C(\tint: L^2).
\]
Therefore, we can prove \eqref{essreg} by showing that
\[
\lim_{t_1\to t}\norm{R_tS\xi(x,t;t_1-t)}{L^2}\to0,
\]
using
\[
\pds{t}s=-u\pds{x}s.
\]
In the sequel, we shall write $\eta=S\xi$, as before, in order to simplify the notation.

Given $x\in\rr$ and $t,t_1\in\tint$, we have
\[
s(x,t_1)=s(x_1,t), \quad\text{with}\quad x_1=x(s(x,t_1),t).
\]
Again, to simplify the notation, we also set 
\begin{equation}
\label{esssimp}
s_1=s(x,t_1)=s(x_1,t), \; s=s(x,t)
\end{equation}
so that
\[
 x_1=x(s_1,t),\; x=x(s,t).
\]
So we can write
\begin{align*}
s(x,t_1)-s(x,t)\,&= s(x_1,t)-s(x,t)\\
&= x_1-x+\eta(x_1,t)-\eta(x,t)\\
&= [(1+\pds{x}\eta(x,t))+R_x\eta(x,t;x_1-x)](x_1-x)\\
&=[\pds{x}s(x,t)+R_x\eta(x,t;x_1-x)](x_1-x).
\end{align*}
Next,  we write
\begin{align*}
x_1-x
&= x(s(x,t_1),t)-x(s(x,t_1),t_1)\\
&= x(s_1,t)-x(s_1,t_1)\\
&= \xi(s_1,t)-\xi(s_1,t_1),
\end{align*}
from which we see that
\[
|x_1-x|\le\norm{\xi(\cdot,t_1)-\xi(\cdot,t)}{L^\infty}\to0,\quad \text{as}\quad t_1\to t.
\]
By \eqref{elltwospder}, it follows that
\[
\norm{R_x\eta(x,t;x_1-x)}{L^2}\to0,\quad \text{as}\quad t_1\to t.
\]
Continuing from above, we have that
\begin{align*}
\xi(s_1,t)-&\xi(s_1,t_1)
=-(\pds{t}\xi(s_1,t)+R_t\xi(s_1,t;t_1-t))(t_1-t)\\
=&-\big(z(s,t)+z(s_1,t)-z(s,t)+R_t\xi(s_1,t;t_1-t)\big)(t_1-t).
\end{align*}
Recalling  the definitions introduced above, we have
\begin{align*}
Z(x,t;t_1-t)&\equiv \frac{x_1-x}{t_1-t}+u(x,t)=\frac{x_1-x}{t_1-t}+z(s,t)\\
&=z(s(x,t_1),t)-z(s(x,t),t)
 +R_t\xi(s_1,t;t_1-t).
\end{align*}
An easy estimation of 
\[
R_t\xi(x,t,t_1-t)=\frac1{t_1-t}\int_t^{t_1} (\pds{t}\xi(x,\tau)-\pds{t}\xi(x,t))d\tau
\]
yields
\begin{equation*}
\norm{Z(\cdot,t;t_1-t)}{L^\infty}
\le C(\norm{z(\cdot,t)}{L^\infty}+\sup_{|\tau-t|\le|t_1-t|}\norm{\pds{t}z(\cdot,\tau)}{L^\infty})\le C.
\end{equation*}

By Corollary \ref{ellpee}, we have
\begin{equation*}
\norm{Z(\cdot,t;t_1-t)}{L^2}\le \norm{z(s(\cdot,t_1),t)-z(s(\cdot,t),t)}{L^2}
+\norm{R_t\xi(\cdot,t;t_1-t)}{L^2}.
\end{equation*}
So, using Lemma \ref{homeo4} and  \eqref{xistrdiff}, we have that
\[
\lim_{t_1\to t}\norm{Z(x,t;t_1-t)}{L^2}=0.
\]

%
%
%
%
%
%

Altogether, we find that
\begin{align*}
\norm{R_ts(\cdot,t;t_1-t)}{L^2}
\le&\norm{R_x\eta(\cdot,t;t_1-t)}{L^2}\norm{u(\cdot,t)}{L^\infty}\\
&+\norm{Z(x,t;t_1-t)}{L^2}\norm{\pds{x}s(\cdot,t)}{L^\infty}\\
&+\norm{R_x\eta(\cdot,t;t_1-t)}{L^2}\norm{Z(x,t;t_1-t)}{L^\infty}\\
\le&\,C(\norm{R_x\eta(\cdot,t;t_1-t)}{L^2}+\norm{Z(x,t;t_1-t)}{L^2}).
\end{align*}
This tends to 0 as $t_1\to t$, thereby proving  the statement \eqref{essreg}.

We now turn to the verification of \eqref{youreg}.  Formally speaking, we expect that
\begin{equation}\label{dtuformula}
\begin{split}
\pds{t}u(x,t)&=\pds{t}z(s(x,t),t)+\pds{s}z(s(x,t),t)\pds{t}s(x,t).
\end{split}
\end{equation}

We first show that the expression on the right is indeed  a function in $C(\tint,L^2)$.

 Using the change of variables in Lemma \ref{cov}, we have
 \begin{equation}\label{dtz}
\begin{split}
\pds{t}z(&s(x,t),t)= \, F_1\xzw(s(x,t),t)\\
=&-\intr G'(x-y))M(z(s(y,t),t),w(s(y,t),t))dy \\
=&-\intr G'(x-y))M(u(y,t),\pds{x}u(y,t))dy\\
=&-G'\ast M(u,\pds{x}u)(x,t)\\
\equiv &-v(x,t).
\end{split}
\end{equation}
By Lemma \ref{mnlip}, the map
\[
u\mapsto M(u,\pds{x}u)
\]
is continuous from $C(\tint: X)$ to $C(\tint: L^2)$.  Convolution with $G'$ is continuous on $L^2$,
so it follows that
\[
v= G'\ast M(u,\pds{x}u)\in C(\tint: L^2).
\]

By \eqref{eikonal} and \eqref{wid},  
we have that
\begin{equation}\label{dsz}
\begin{split}
\pds{s}z(s(x,t),t)\;\pds{t}s(x,t)
&=\pds{s}z(s(x,t),t) \; (u(x,t) \pds{x}s(x,t))\\
&= u(x,t)\;\pds{x}u(x,t).
\end{split}
\end{equation}
Since $u\in C(\tint: L^\infty)$ and $\pds{x}u\in C(\tint: L^2)$, we have that
$u\pds{x}u\in C(\tint: L^2)$.  

Therefore, in order to prove \eqref{youreg}, it is enough to show that
\[
\lim_{t_1\to t}\norm{R_tu(x,t;t_1-t)}{L^2}=0,
\]
with 
$\pds{t}u=-u\pds{x}u-v$.
By \eqref{dtz} and \eqref{dsz}, this will also rigorously establish \eqref{dtuformula}.

We shall examine the expression
\begin{equation}\label{dumain}
\begin{split}
u(x,t_1)-u(x,t)=& z(s(x,t_1),t_1)-z(s(x,t),t)\\
=& z(s_1,t_1)-z(s,t)\\
=& (z(s_1,t_1)-z(s_1,t))+(z(s_1,t)-z(s,t)),
\end{split}
\end{equation}
where we have used the notation \eqref{esssimp}.

Since $z\in C^1(\tint: X)$, the first pair of terms in \eqref{dumain} has the form
\begin{equation}\label{du1}
\begin{split}
&z(s_1,t_1)-z(s_1,t)\\
&= \, \pds{t}z(s,t)(t_1-t)+\int_t^{t_1}(\pds{t}z(s_1,\tau)-\pds{t}z(s,\tau))d\tau\\
&\hskip15pt+\int_t^{t_1}(\pds{t}z(s,\tau)-\pds{t}z(s,t))d\tau\\
&=\Big\{\pds{t}z(s,t)
+\frac1{t_1-t}\int_t^{t_1}\int_s^{s_1}\pds{s}\pds{t}z(\sigma,\tau)d\sigma d\tau\\
&\hskip15pt+\frac1{t_1-t}\int_t^{t_1}(\pds{t}z(s,\tau)-\pds{t}z(s,t))d\tau\Big\}(t_1-t).
\end{split}
\end{equation}
Since by \eqref{esssimp}
\begin{align*}
s_1-s=&(\pds{t}s(x,t)+R_ts(x,t;t_1-t))(t_1-t),
\end{align*}
the second group of terms in \eqref{dumain} can be written as
\begin{equation}\label{du2}
\begin{split}
z(s_1,t)-z(s,t)=&(\pds{s}z(s,t)+R_sz(s,t;s_1-s))(s_1-s)\\
=&(\pds{s}z(s,t)+R_sz(s,t;s_1-s)\\
&\times (\pds{t}s(x,t)+R_ts(x,t;t_1-t))(t_1-t).
\end{split}
\end{equation}

Using \eqref{dtz} and \eqref{dsz}, we obtain
\begin{equation}\label{du3}
\begin{split}
R_tu(x,t;t_1-t)
=&\frac{u(x,t_1)-u(x,t)}{t_1-t}+v(x,t)+u\pds{x}u(x,t)\\
=&\frac{u(x,t_1)-u(x,t)}{t_1-t}\\
&-\pds{t}z(s(x,t),t)-\pds{s}z(s(x,t),t)\;\pds{t}z(s,t)
\end{split}
\end{equation}
We now combine \eqref{dumain}, \eqref{du1}, \eqref{du2}, \eqref{du3} to derive
\[
R_tu(x,t;t_1-t)=A_1+\ldots+A_5,
\]
where
\begin{align*}
&A_1=\frac1{t_1-t}\int_t^{t_1}\int_s^{s_1}\pds{s}\pds{t}z(\sigma,\tau)d\sigma d\tau\\
&A_2=\frac1{t_1-t}\int_t^{t_1}(\pds{t}z(s,\tau)-\pds{t}z(s,t))d\tau\\
&A_3= \pds{t}s(x,t)\;R_sz(s,t;s_1-s)\\
&A_4=\pds{s}z(s,t)\;R_ts(x,t;t_1-t)\\
&A_5= R_sz(s,t;s_1-s)\;R_ts(x,t;t_1-t).
\end{align*}
Therefore, the task of proving \eqref{youreg} reduces to showing
\begin{equation}
\label{task}
\lim_{t_1\to t}\norm{A_j}{L^2}=0, \quad j=1,\ldots,5.
\end{equation}

Since $s\in C(\tint: L^2)$,
\[
\norm{A_1}{L^2}\le C\norm{s(\cdot,t_1)-s(\cdot,t)}{L^2}\to0,
\]
as $ t_1\to t$.

By Corollary \ref{ellpee}, we have
\begin{align*}
\norm{A_2}{L^2}&\le \sup_{|\tau-t|\le|t_1-t|}\norm{\pds{t}z(s(\cdot,t),\tau)-\pds{t}z(s(\cdot,t),t)}{L^2}\\
&\le C(\rho)\sup_{|\tau-t|\le|t_1-t|}\norm{\pds{t}z(\cdot,\tau)-\pds{t}z(\cdot,t)}{L^2},
\end{align*}
which tends to 0, as $t_1\to t$, since $\pds{t}z\in C(\tint: L^2)$.

Since $u  \in C(\tint: H^1)\subset C(\tint: L^\infty)$, by the Sobolev Lemma, we have that $u$ is uniformly bounded.
By \eqref{ellinfbd}, we have that $\pds{x}s=1+\pds{x}\xi$ is uniformly bounded.  
Therefore, by
\eqref{eikonal}, we have that $\pds{t}s=-u\pds{x}s$ is uniformly bounded.
By Corollary \ref{ellpee} and \eqref{elltwospder}, we have that
\[
\norm{R_sz(s(\cdot,t),t;h)}{L^2}\le C(\rho) \norm{R_sz(\cdot,t;h)}{L^2}\to0,
\]
as $h\to0$.  Since $S\xi\in C(\tint: L^\infty)$, we have that 
\[
\norm{s_1-s}{L^\infty}=\norm{s(\cdot,t_1)-s(\cdot,t)}{L^\infty}=\norm{S\xi(\cdot,t_1)-S\xi(\cdot,t)}{L^\infty}\to0,
\]
as $t_1\to t$.  It follows that
\[
\lim_{t_1\to t}\norm{R_sz(s(\cdot,t),t;s_1-s)}{L^2}=0.
\]
Therefore, we have shown that \eqref{task} is valid for $j=3$.

Since $\pds{s}z\in C(\tint: Y)\subset C(\tint: L^\infty)$, we have that $\pds{s}z$ is uniformly bounded.
In proving \eqref{essreg}, we showed that 
\[
\lim_{t_1\to t}\norm{R_ts(x,t;t_1-t)}{L^2}=0.
\]
Thus, we obtain the desired conclusion \eqref{task} for $j=4$.

Finally, since  $\pds{t}s$ is uniformly bounded, as noted above, and since $s\in C^1(\tint: L^2)$, we have that
\begin{equation*}
\begin{split}
|R_ts(x,t;t_1-t)|&=\left|\int_t^{t_1}(\pds{t}s(x,\tau)-\pds{t}s(x,t))d\tau\right|\\
&\le C\norm{\pds{t}z}{L^\infty(\rr\times\tint)}\le C.
\end{split}
\end{equation*}
So using \eqref{elltwospder}, we have  that \eqref{task} holds for $j=5$.

This concludes the proof of \eqref{youreg} and the  theorem.

\end{proof}

Now we establish uniqueness of solutions in the Eulerian frame.
\begin{theorem}
\label{euniqueness}
If 
\[
u\in C(\tint: H^1)\cap L^\infty(\tint: W^{1,\infty})\cap C^1(\tint: L^2)
\]
solves the initial value problem \eqref{CH-char} with $u_0\in X$, then
$u$ is unique.
\end{theorem}

\begin{proof}
The function $u$ is continuous on $\rr\times\tint$, and $u(\cdot,t)$ is Lipschitz for each $t\in\tint$.
By the existence and continuous dependence theorems for ODEs, the
problem
\begin{equation}
\label{xiode}
\pds{t}\xi(s,t)=u(s+\xi(s,t),t),\quad \xi(s,0)=0
\end{equation}
has a unique solution $\xi\in C^1(\tint:  C(\rr))$.

Since $u\in L^\infty(\tint : W^{1,\infty})$, we have that, for $h\ne0$,
\begin{align*}
&\left|h^{-1}[\xi(s+h,t)-\xi(s,t)]\right|\\
&\qquad=\left|h^{-1}\int_0^t[u(s+h+\xi(s+h,\tau),\tau)-u(s+\xi(s,\tau),\tau)]d\tau\right|\\
&\qquad\le \left|\int_0^t \norm{\pds{x}u(\cdot,\tau)}{L^\infty}[1+\left|h^{-1}[\xi(s+h,\tau)-\xi(s,\tau)]\right|d\tau\right|.
\end{align*}
It follows from Gronwall's inequality that
\begin{equation}
\label{xilip}
\left|h^{-1}[\xi(s+h,t)-\xi(s,t)]\right|\le C\left(\int_{-T}^T\norm{\pds{x}u(\cdot,\tau)}{L^\infty}d\tau\right),
\end{equation}
which proves that $\xi\in L^\infty(\tint : W^{1,\infty})$.
From \eqref{xiode}, we can now also say that 
\[
\pds{t}\xi\in C(\tint:  C(\rr))\cap L^\infty(\tint: W^{1,\infty}).
\]

Now define
\begin{align*}
x(s,t)&=s+\xi(s,t)\\
 z(s,t)&=u(x(s,t),t)\\
 w(s,t)&=\pds{x}u(x(s,t),t)\\
 y(s,t)&=\exp \int_0^t w(s,\tau)d\tau,
\end{align*}
for $(s,t)\in\rr\times\tint$.
Note that
\begin{gather}
\label{xode}
\pds{t}x=\pds{t}\xi=z
\intertext{and}
\label{wode}
\pds{t}y=wy.
\end{gather}
These functions  have the following regularity properties:
\begin{gather*}
z\in C(\tint:  C(\rr))\cap L^\infty(\tint: W^{1,\infty}),\\
\pds{s}z,w\in L^\infty(\tint: L^\infty),\\
y\in C(\tint: L^\infty),\\
\pds{t}y\in L^\infty(\tint: L^\infty).
\end{gather*}

We claim that 
\begin{equation}
\label{dsxid}
\pds{s}x(s,t)=y(s,t).
\end{equation}
A simple calculation gives
\begin{multline}
\label{rsxid}
\pds{t} R_sx(s,t;h)
=w(s,t)R_sx(s,t;h)\\
+R_xu(x(s,t),t;\Delta_s x(s,t;h))\;\Delta_s x(s,t;h)/h,
\end{multline}
in which (as in the proof of Theorem \ref{regthm})
\begin{gather*}
R_sx(s,t;h)=
\begin{cases}
h^{-1}[x(s+h,t)-x(s,t)]-y(s,t), & h\ne0\\
0,&h=0,
\end{cases}\\
\ \\
R_xu(x,t;k)=
\begin{cases}
k^{-1}[u(x+k,t)-u(x,t)]-\pds{x}u(x,t),& k\ne0\\
0,&k=0,
\end{cases}
\intertext{and}
\Delta_s x(s,t;h)=x(s+h,t)-x(s,t).
\end{gather*}
Since $w\in L^\infty$, there exist constants $\rho$, $C$ such that
\[
0<2\rho< y(s,t)=\exp  \int_0^t w(s,\tau)d\tau\le C,\quad \text{a.e.\ s,}
\]
 for all $t\in\tint$, and 
 \[
 R_sx(0,t;h)=0,
 \]
  it follows from \eqref{rsxid} and Gronwall's inequality that
\begin{equation*}
\norm{ R_sx(\cdot,t;h)}{L^\infty}\!\!
\le \!\!
\left|\int_0^t \!\!\norm{R_xu(x(\cdot,\tau),\tau;\Delta_s x(\cdot,\tau;h))\;\Delta_s x(\cdot,\tau;h)/h}{L^\infty}
d\tau\right|.
\end{equation*}
By \eqref{xilip}, we see that $\Delta_sx(s,t;h)=O(h)$, uniformly for $(s,t)\in \rr\times\tint$.
Therefore, since $u\in L^\infty(\tint: W^{1,\infty})$, we also have that 
\[
R_xu(x(s,t),t;\Delta_s x(s,t;h))=O(h),
\]
uniformly for $(s,t)\in \rr\times\tint$.  As a consequence, we conclude that
\[
\lim_{h\to0}\norm { R_sx(\cdot,t;h)}{L^\infty}=0,
\]
which verifies the claim.

Having verified that 
\[
1+\essinf\pds{s}\xi(\cdot,t)=\essinf \pds{s}x(\cdot,t)=\essinf y(\cdot,t)>\rho,
\]
we can apply Corollary \ref{ellpee} to obtain
\[
z\in L^\infty(\tint: H^1)
\quad \text{and}\quad
w\in L^\infty(\tint: L^2).
\]
Thus, we also get that
\[
\xi(s,t)=\int_0^tz(s,\tau)d\tau\in L^\infty(\tint: H^1).
\]
Thus, $(\xi,z,w)(\cdot,t)\in\oo\times X\times Y$, for $t\in\tint$.

By the chain rule, we have
\[
\pds{t}z(s,t)=(\pds{t}u+u\pds{x}u)(x(s,t),t),
\]
and  since $u$ solves \eqref{CH-char}, we see that
\begin{align*}
z(s,t)&=z(s,0)+\int_0^t\pds{t}z(s,\tau) d\tau\\
&=u_0(s)-\int_0^tG'\ast M(u,\pds{x}u)(x(s,\tau),\tau)d\tau.
\end{align*}
Using Lemma \ref{ellone}, we can change variables to get
\begin{equation}
\label{zeeinteq}
z(s,t)=u_0(s)+\int_0^tF_1(\xi,z,w)(s,\tau)d\tau.
\end{equation}
Therefore,
\[
\pds{t}z=F_1(\xi,z,w),
\]
and then by Lemma \ref{harddiff}, 
\begin{equation}
\label{dtdszee}
\pds{t}\pds{s}z=(F_2(\xi,z,w)+w^2)\pds{s}x.
\end{equation}
From \eqref{wode}, \eqref{dsxid}, and \eqref{xode}, we see that
\[
w=\frac{\pds{t}y}{y}=\frac{\pds{t}\pds{s}x}{y}=\frac{\pds{s}z}{y}.
\]
From \eqref{dtdszee}, \eqref{wode}, and \eqref{dsxid}, this yields
\[
\pds{t}w=\frac{\pds{t}\pds{s}z}{y}-\frac{\pds{s}z\pds{t}y}{y^2}=(F_2(\xi,z,w)+w^2)-w^2=F_2(\xi,z,w).
\]
Observing that 
\[
w(s,0)=\pds{x}u(x(s,0),0)=u_0'(x(s,0))=u_0'(s),
\]
we find that
\begin{equation}
\label{winteq}
w(s,t)=u_0'(s)+\int_0^tF_2(\xi,z,w)(s,\tau)d\tau,\quad t\in\tint.
\end{equation}

Combining \eqref{xode}, \eqref{zeeinteq}, and \eqref{winteq}, we obtain that
\[
\vv(s,t)=(\xi(s,t),z(s,t),w(s,t))\in C(\tint: \oo\times X\times Y)
\]
solves the integral equation
\[
\vv(s,t)=\vv(s,0)+\int_0^t\fff(\vv(s,\tau))d\tau,\quad t\in\tint,
\]
where $\fff$ was defined and shown to be locally Lipschitz on $\oo\times X\times Y$ in Theorem \ref{effloclip}.
Thus, we can
use Gronwall's inequality to see that $\vv$ is the unique solution on $\tint$.  Since $x(\cdot,t)$
is a homeomorphism for every $t\in\tint$, by Lemma \ref{homeo}, we conclude that $u$ is unique,
as well.

\end{proof}

\subsection{Continuous dependence on initial conditions}

\begin{theorem}\label{contdep}
Let $\eps>0$  be given.
Fix initial data $u_{01}\in X$, and choose $B>\norm{u_{01}}{X}$.
There exists a $0<\delta<B-\norm{u_{01}}{X}$ such that for all
\[
u_{02}\in  \neigh_\delta=\{v\in X:\norm{u_{01}-v}{X}<\delta\}
\]
the corresponding solutions 
\[
u_j\in C(\tints{j}: H^1)\cap C^1(\tints{j}: L^2),\quad j=1,2,
\]
 of the initial value problem \eqref{CH-char} constructed in Theorem \ref{regthm} satisfy
\[
\sup_{|t|\le T}(\norm{u_1(\cdot,t)-u_2(\cdot,t)}{H^1}+\norm{\pds{t}u_1(\cdot,t)-\pds{t}u_2(\cdot,t)}{L^2})<\eps,
\]
where $T=\min\{T_1,T_2\}$.
\end{theorem}

\begin{proof}
In this proof, generic constants may depend on $\rho$, $B$, and $T=\min\{T_1,T_2\}$.

Given $u_{01}\in X$ and $u_{02}\in  \neigh_\delta$, with $\delta<B-\norm{u_{01}}{X}$,
define
\[
\vv_{0j}=(
0,u_{0j},u'_{0j}).
\]
Since $\norm{\vv_{0j}}{\xx}=\norm{u_{0j}}{X}<B$, we see that 
\[
\vv_{0j}\in \neigh(\rho, B)=\{\vv=(\xi,z,w)\in\xx:\xi\in\oo,\; \norm{\vv}{\xx}\le B\}.
\]
Let 
\[
\vv_j=
(\xi_j,u_j,w_j)
\in C^1(\tint:   \neigh(\rho, B))
\]
 be the corresponding solutions from Theorem  \ref{localexistence}.
 This means that 
 \begin{equation}
\label{mainbd0}
\sup_{|t|\le T}
\Bigg(\norm{\pds{t}^k\xi_j(\cdot,t)}{X}
+\norm{\pds{t}^kz_j(\cdot,t)}{X}
+\norm{\pds{t}^kw_j(\cdot,t)}{Y}\Bigg)<B
\end{equation}
and $\xi_j(\cdot,t)\in\oo$, $t\in\tint$, for $j=1,2$ and $k=0,1$.
 
Since the vector field $\fff$ is locally Lipschitz, we obtain from Gronwall's inequality that
\[
\sup_{|t|\le T}\norm{\vv_1(\cdot,t)-\vv_2(\cdot,t)}{\xx}\le C\norm{\vv_{01}-\vv_{02}}{\xx}=C\norm{u_{01}-u_{02}}{X}<C\delta.
\]
This implies that
\begin{equation}\label{mainbd}
\begin{split}
\sup_{|t|\le T} \sum_{k=1,2}\Bigg(\|\pds{t}^k\xi_1(\cdot,t)&-\pds{t}^k\xi_2(\cdot,t)\|_{X}
+\norm{\pds{t}^kz_1(\cdot,t)-\pds{t}^kz_2(\cdot,t)}{X}\\
&+\|\partial_t^kw_1(\cdot,t)-\pds{t}^kw_2(\cdot,t)\|_{Y}\Bigg)< C\delta.
\end{split}
\end{equation}

By Lemma \ref{ycont} and \eqref{mainbd}, we have
\begin{equation}
\label{mapcont}
\begin{split}
\sup_{|t|\le T}\|S\xi_1&(\cdot,t)-S\xi_2(\cdot,t)\|_{L^2}\\
&\le C\sup_{|t|\le T}\norm {\xi_1(\cdot,t)-\xi_2(\cdot,t)}{L^2}
\le C\delta. 
\end{split}
\end{equation}

The rest of the proof will rely on the following statement.
\bigskip

\begin{claim}
\label{claim}
Define the set
\[
 \neigh=\{\xi\in\oo:\norm{\xi}{X}<B\}.
\]
Let 
$
f\in C(\tint: L^2).
$
The map
\[
\xi(\cdot,t)\mapsto f(s(\cdot,t),t)
\]
is uniformly continuous from
$C(\tint:  \neigh)$
into
 $ L^\infty(\tint: L^2)$.
 \end{claim}

To prove the Claim \ref{claim}, let $\eps'>0$ be given.  Since $\tint$ is compact
and  $f\in C(\tint: L^2)$,
the map 
\[
t\mapsto f(\cdot,t)
\]
is uniformly continuous from $\tint$ to $L^2$.  So we may choose $\alpha>0$ such that
\[
t',t\in\tint,\;|t'-t|<\alpha\quad\text{implies}\quad \norm{f(\cdot,t')-f(\cdot,t)}{L^2}<\eps'/4C_0,
\]
where the constant $C_0$ will be defined below.

Define a partition
\[
t_k=-T+k\Delta t,\quad k=0,1,\ldots,n,
\]
where $\Delta t=2T/n<\alpha$.  Then for any  $t\in\tint$, there exists a $t_k$ such that
$|t-t_k|<\alpha$.

Let  $\xi_1,\xi_2\in C(\tint:  \neigh)$.  
By Lemma \ref{homeo4}, there exists $\delta'>0$ such that the inequality
\begin{equation}
\label{term2}
\sup_{|t|\le T}\norm{f(s_1(\cdot,t),t_k)-f(s_2(\cdot,t),t_k)}{L^2}<\eps'/2,\quad k=0,1,\ldots,n.
\end{equation}
holds, provided that
\begin{equation}
\label{deltaneigh}
\sup_{|t|\le T}\norm{\xi_1(\cdot,t)-\xi_2(\cdot,t)}{L^2}<\delta'.
\end{equation}

Now, with $t\in\tint$ fixed, we have
\begin{equation}\label{mainest}
\begin{split}
\|f(s_1(\cdot,t),t)&-f(s_2(\cdot,t),t)\|_{L^2}\\
\le&  \norm{f(s_1(\cdot,t),t)\!-\!f(s_1(\cdot,t),t_k)}{L^2}\\
&+\norm{f(s_1(\cdot,t),t_k)-f(s_2(\cdot,t),t_k)}{L^2}\\
&+\norm{f(s_2(\cdot,t),t_k)-f(s_2(\cdot,t),t)}{L^2}.
\end{split}
\end{equation}
By Corollary \ref{ellpee}, there exists a constant $C_0$ depending only on $B$ such that
\begin{equation}\label{term13}
\begin{split}
\|f(s_j(\cdot,t),t)&-f(s_j(\cdot,t),t_k)\|_{L^2}\\
&\le C_0\norm{f(\cdot,t)-f(\cdot,t_k)}{L^2}\\
&<\eps'/4,\hskip30pt\text{for}\quad j=1,2;\quad k=0,1,\ldots,n.
\end{split}
\end{equation}

Therefore, if \eqref{deltaneigh} holds, then the estimates \eqref{mainest}, \eqref{term13}, \eqref{term2} imply that
\[
\norm{f(s_1(\cdot,t),t)-f(s_2(\cdot,t),t)}{L^2}<\eps',\quad\text{for all}\quad  t\in\tint.
\]
This completes the proof of the Claim \ref{claim}.

\bigskip

By \eqref{mainbd0}, we see that $\xi_j\in C(\tint: \neigh)$.

By the definition of $u_j$, we have
\begin{equation*}
\begin{split}
&\norm {u_1(\cdot,t)-u_2(\cdot,t)}{L^2}\\
&\le\norm{z_1(s_1(\cdot,t),t)-z_1(s_2(\cdot,t),t)}{L^2}
+\norm{z_1(s_2(\cdot,t),t)-z_2(s_2(\cdot,t),t)}{L^2}.
\end{split}
\end{equation*}
Since $u_1\in C(\tint: L^2)$, the Claim \ref{claim} and  \eqref{mapcont} imply that the first term satisfies
\[
\norm{z_1(s_1(\cdot,t),t)-z_1(s_2(\cdot,t),t)}{L^2}=o(\delta).
\]
From Corollary \ref{ellpee}, the second term is estimated by
\[
\norm{z_1(s_2(\cdot,t),t)-z_2(s_2(\cdot,t),t)}{L^2}\le C\norm {z_1(\cdot,t)-z_2(\cdot,t)}{L^2}
\le C\delta.
\]
This proves that 
\[
\sup_{|t|\le T}\norm{u_1(\cdot,t)-u_2(\cdot,t)}{L^2}=o(\delta).
\]

Since $\pds{x}u_j(\cdot,t)=w_j(s_j(\cdot,t),t)$ and $w_j\in C(\tint;L^2)$, the same argument as above  yields
\[
\sup_{|t|\le T}\norm{\pds{x}u_1(\cdot,t)-\pds{x}u_2(\cdot,t)}{L^2}=o(\delta).
\]
Therefore, we have that
\begin{equation}
\label{honecontdept}
\sup_{|t|\le T} \norm{u_1(\cdot,t)-u_2(\cdot,t)}{H^1}=o(\delta).
\end{equation}

For the time derivatives, we use the PDE \eqref{CH-char} to obtain
\begin{multline*}
\norm{\pds{t}u_1(\cdot,t)-\pds{t}u_2(\cdot,t)}{L^2}\le
\norm {u_1\pds{x}u_1(\cdot,t)-u_2\pds{x}u_2(\cdot,t)}{L^2}\\
+\norm{G'\ast M(u_1,\pds{x}u_1)(\cdot,t)-G'\ast M(u_2,\pds{x}u_2)(\cdot,t)}{L^2}.
\end{multline*}
These terms are easily estimated using the Young inequality, Lemma \ref{mnlip}, Theorem \ref{regthm}, and \eqref{honecontdept}, with the result that
\[
\sup_{|t|\le T}\norm{\pds{t}u_1(\cdot,t)-\pds{t}u_2(\cdot,t)}{L^2}=o(\delta).
\]
This completes the proof. 
\end{proof}

Taken together, Theorems \ref{localexistence}, \ref{regthm}, \ref{euniqueness}, and \ref{contdep}
imply Theorem \ref{thm1}.

\section{Propagation of Regularity}
\begin{proof}[{\bf Proof of Theorem \ref{thm2} part (a)}]

We consider  $\rho \in C^{\infty}_0(\R),$ such that
$$
\rho(x)\geq 0,\;\;\;  \,\supp (\rho)\subset (-1,1),\;\;\;\int \rho(x)dx =1,
$$
and define $\rho_{\eps}(x)=\frac{1}{\eps} \rho(\frac{x}{\eps})$ for $ \eps\in (0,1)$. For   $u_0\in X$ let
$$
u_0^{\eps}(x)=\rho_{\eps}\ast u_0(x).
$$
Thus, $u_0^{\eps}\in H^s(\R)$ for any $s\in \R$
and  for $\eps\in(0,1)$
$$
\|u_0^{\eps}\|_X\leq \|u_0\|_X.
$$
 We denote by $u^{\eps}=u^{\eps}(x,t)$ the corresponding solution of the IVP associated to the RCH equation \eqref{RCH} provided by Theorem \ref{thm1}. By Theorem \ref{thm1} there exist $K>0$ and  $T=T(\|u_0\|_X)>0$ (independent of $\eps\in(0,1)$) such that
\begin{equation}\label{bound}
\sup_{\eps \in (0,1)}\,\sup_{t\in[-T,T]}\|u^{\eps}(t)\|_X\leq 2 c \|u_0\|_X=K.
\end{equation}
In particular,
$$
\int_{-T}^T\|\partial_xu^{\eps}(\cdot, t)\|_{\infty}\,dt < c,
$$
with $c$ independent of $\eps$. Using the \it a priori \rm energy estimate, see \cite{KP}, one has that for  any $\tilde T>0$ and any $s>0$
$$
\sup_{t\in[-T,\tilde T]}\|u^{\eps}(\cdot,t)\|_{s,2}\leq c_s\|u_0^{\eps}\|_{s,2}\,\exp{(\int_{-\tilde T}^{\tilde T}\|\partial_xu^{\eps}(\cdot, t)\|_{\infty}dt)},
$$
this combined with \eqref{bound} allows us to extend the solution $u^{\eps}$ of the IVP for the RCH equation \eqref{RCH} obtained in \cite{LO} and in \cite{RB} to the time interval $[-T,T]$ such that 
$$
u^{\eps}\in C([-T,T] : H^s(\R)),\;\;\;\;\text{for any}\;\;\;\; s\in \R.
$$

To simplify the exposition we shall assume that 
$$
M(u^{\eps})=(\partial_xu^{\eps})^2,\;\;\;\;\text{and}\;\;\;\;\Omega =(a,b),\;\,\,\;\;a,\,b\in \R,\;\,a<b.
$$
  It will be clear from our argument that this does not entail a loss of generality.
\vskip.1in
\underline {Case $j=2$: }
\vskip.1in
 For each $\eps\in (0,1)$ the function $\partial_x^2u^{\eps}$ satisfies the equation
\begin{equation}\label{eq1}
\partial_t\partial_x^2u^{\eps}+u^{\eps}\partial_x\partial_x^2u^{\eps}+c_2\partial_xu^{\eps}\partial_x^2u^{\eps}
-\partial_xG\ast M(u^{\eps})=0.
\end{equation}

By hypothesis it follows that 
$$
u_0^{\eps}|_{(a+\eps,b-\eps)}\in H^{2,p}((a+\eps,b-\eps)),
$$ 
with norm independent of $\eps$.
Multiplying the equation \eqref{eq1} by 
$$
 p\, |\partial_x^2u^{\eps}(x(s,t),t)|^{p-1}\sgn(\partial_x^2u^{\eps}(x(s,t),t)),
$$
and using that
$$
\partial_t(\partial_x^2u^{\eps}(x(s,t),t))=\partial_t\partial_x^2u^{\eps}(x(s,t),t)+u\partial_x\partial_x^2u^{\eps}(x(s,t),t)
$$
one gets 
\begin{equation}\label{eq2}
\begin{aligned}
& \frac{d}{dt} \int_{a+\eps}^{b-\eps} |\partial_x^2u^{\eps}(x(s,t),t)|^pds\\
&\leq c_2 \int_{a+\eps}^{b-\eps} |\partial_xu^{\eps}\partial_x^2u^{\eps}(x(s,t),t)|   |\partial_x^2u^{\eps}(x(s,t),t)|^{p-1}ds\\
&\hskip10pt +\int_{a+\eps}^{b-\eps}|\partial_xG\ast (\partial_xu)^2(x(s,t),t)| |\partial_x^2u^{\eps}(x(s,t),t)|^{p-1}ds\\
&=A_1+A_2.
\end{aligned}
\end{equation}
One sees that
$$
A_1\leq \|\partial_xu^{\eps}(t)\|_{\infty} \int_{a+\eps}^{b-\eps} |\partial_x^2u^{\eps}(x(s,t),t)|^pds,
$$
and using  that quantity $|\partial x/\partial s\,(s,t)|$ is bounded above and below uniformly in $t\in[-T,T]$ and $\eps\in (0,1)$ that
\[
\begin{aligned}
A_2&\leq c\|\partial_xG\|_p\|(\partial_xu^{\eps})^2\|_1(\int_{a+\eps}^{b-\eps} |\partial_x^2u^{\eps}(x(s,t),t)|^pds)^{(p-1)/p}\\
&\leq c\|\partial_xu^{\eps}\|_2^2(\int_{a+\eps}^{b-\eps} |\partial_x^2u^{\eps}(x(s,t),t)|^pds)^{(p-1)/p}.
\end{aligned}\]
Inserting these estimates in \eqref{eq2} it follows that
\begin{equation}\label{eq3}
\begin{aligned}
& \frac{d}{dt}( \int_{a+\eps}^{b-\eps} |\partial_x^2u^{\eps}(x(s,t),t)|^pds)^{1/p}\\
&\leq c\| \partial_x u^{\eps}(t)\|_{\infty} ( \int_{a+\eps}^{b-\eps} |\partial_x^2u^{\eps}(x(s,t),t)|^pds)^{1/p}
+ c\|\partial_xu^{\eps}\|_2^2.
\end{aligned}
\end{equation}

From \eqref{eq3} one concludes, using \eqref{bound}, that for any $\eps\in(0,1)$
$$
\sup_{[-T,T]}\big( \int_{a+\eps}^{b-\eps} |\partial_x^2u^{\eps}(x(s,t),t)|^p\,ds\big)^{1/p}
\leq C(\|u_0\|_X;(\int_a^b|\partial_x^2u_0(x)|^pdx)^{1/p}).
$$
Combining this estimate with the fact, provided by Theorem \ref{thm1}, that $u^{\eps}\to u$ in $C([-T,T]:H^1(\R))$ and some weak covergence arguments it follows that
\[
\underset{[-T,T]}{\sup} (\int_{a}^{b}\!\! |\partial_x^2u(x(s,t),t)|^pds)^{1/p}\!  \leq \!C(\|u_0\|_X; (\int_a^b\!\!|\partial_x^2u_0(x)|^pdx)^{1/p}).
\]
Finally, using that the quantity $|\partial x/\partial s(s,t)|$ is bounded above and below uniformly in $t\in[-T,T]$ one obtains the desired result
\[
\sup_{[-T,T]}( \int_{a(t)}^{b(t)} |\partial_x^2u(x,t)|^pds)^{1/p}  \leq C(\|u_0\|_X;(\int_a^b|\partial_x^2u_0(x)|^pdx)^{1/p}),
\]
where
$$
a(t)=x(a,t),\;\;\;\;\;\text{and}\;\;\;\;\;\;b(t)=x(b,t).
$$
\vskip.1in
\underline {Case $j=3$: }
\vskip.1in
 Using that 
$$
\partial_x^2G\ast f= G\ast f-f,
$$
it follows that for each $\eps\in (0,1)$ the function $\partial_x^3u^{\eps}$ satisfies the equation
\begin{equation}\label{eqb1}
\begin{aligned}
&\partial_t\partial_x^3u^{\eps}+u^{\eps}\partial_x\partial_x^3u^{\eps}+c^1_3\partial_xu^{\eps}\partial_x^3u^{\eps}\\
&+c^2_3\partial_x^2u^{\eps}\partial_x^2u^{\eps}
+(\partial_xu^{\eps})^2-G\ast M(u^{\eps})=0.
\end{aligned}
\end{equation}
We observe that reapplying the previous argument for the case $j=2$ one can handle the first three terms in \eqref{eqb1}. The fourth term in
 \eqref{eqb1} was estimated in the previous step, and the final two term are bounded when estimated in the $L^p(\R)$-norm. Hence, reapplying the argument given in the case $j=2$ one gets the desired result.
\vskip.1in
\underline {Case $j>3$: }
\vskip.1in
Writing the equation for $\partial_x^ju^{\eps}$ one observes that this has three kind of terms, the first ones appear as :
$$
\partial_t\partial_x^ju^{\eps}+u^{\eps}\partial_x\partial_x^ju^{\eps}+c^1_j\partial_xu^{\eps}\partial_x^ju^{\eps}
$$
which can be estimated using the argument provided in detail for the case $j=2$, the second ones are terms involving 
derivatives of order less than $j$, i.e. product of terms of the form
$$
\partial_x^lu^{\eps}\,\;\;\;\;\;\;\text{with}\;\;\;\;\;l=2,...,j-1,
$$
which have been previously estimated, and finally each term in the third group have one of the following form
$$
(\partial_xu^{\eps})^2,\;\;\;\;\;\;\;\;\;\;G\ast M(u^{\eps}),\;\;\;\;\;\partial_xG\ast M(u^{\eps})
$$
which can be estimated in the whole real line. Hence, the result for the general case follows the argument previously described.
\end{proof}

\begin{proof}[{\bf Proof of Theorem \ref{thm2} part (b)}]

To simplify the exposition  as in the proof of part (a) we shall assume 
$$
M(u)=(\partial_xu)^2\;\;\;\;\;\;\text{and}\;\;\;\;\;\Omega=(a,b).
$$ 
It will be clear from our proof below that these assumptions do not represent any loss of generality.

\vskip.1in
\underline{Case:} $\;\partial_xu_0|_{(a,b)}\in C(a,b)$.
\vskip.1in
 From the proof of Theorem \ref{thm1} in Lagrangian coordinates we have
 \[
 w(s,t)=w_0(s)- \int_0^tw^2(s,\tau)d\tau +q(s,t),
\]
 where
\[
\label{w}
w(s,t)=\partial_xu(x(s,t),t),
\]
and
 \begin{equation}
 \label{aa2}
 q(s,t)=\int_0^t\int_{-\infty}^{\infty} G(x(s,\tau)-y)w^2(s(y,\tau),\tau)dy\,d\tau.
 \end{equation}

Since 
$$
w\in C([-T,T] : L^2(\R)\cap L^{\infty}(\R)),
$$
it follows  that
\begin{equation}
 \label{aa3}
 \begin{aligned}
q= q(s,t)\in & \, C^1([-T,T]: X)\\
\equiv &\, C^1([-T,T]:H^1(\R)\cap W^{1,\infty}(\R)).
\end{aligned}
 \end{equation}
 
We recall the hypothesis  $w_0|_{(a,b)}\in C(a,b)$ and consider the sequence
\begin{equation}
\label{aa4}
w_{n+1}(s,t)=w_0(s)+\int_0^t w^2_n(s,\tau)d\tau +q(s,t),\;\;\;\;\;\;\;\;\;\;n\in \Z^+.
\end{equation}
with
$$
w_1(x,t)=w_0(s)+q(s,t)\in C([-T,T]:C(a,b)).
$$

Thus, $w_n \in C([-T,T]:C(a,b))$ for each $n\in\Z^+$ with $w_n$ converging to $w=w(x,t)$ in the $L^{\infty}(\R\times [-T,T])$-norm. Hence,
\[
\label{aa5}
w\in C([-T,T]:C(a,b)).
\]

Since
\[
\partial_x u(x,t)=w(s(x,t),t),
\]
with $s=s(x,t)$ the inverse of 
\begin{equation}
\label{aa6a}
x(s,t)=s+\int_0^tz(s,\tau) d\tau,
\end{equation}
thus
$$
x(s,t)-s=\int_0^t z(s,\tau)d\tau\in C^1([-T,T]:X),
$$
and it follows that
\[
\partial_xu\in C([-T,T]: C(x(a,t),x(b,t)))
\]
which yields the desired result.

\vskip.1in
\underline{Case:} $\,\partial_xu_0|_{(a,b)}\in C^{\theta}(a,b),\;\theta\in(0,1)$.
\vskip.1in
Since by hypothesis $w_0=w_0(s) \in C([-T,T]:C^{\theta}(a,b))$ from \eqref{aa3} one gets that
$$
w_1(x,t)=w_0(s)+q(s,t)\in C([-T,T]:C^{\theta}(a,b)),
$$
and from \eqref{aa4} that  $w_n \in C([-T,T]:C^{\theta}(a,b))$ for each $n\in\Z^+$ with $w_n$ converging in the $L^{\infty}([-T,T]:C^{\theta}(a,b))$-norm, since

\[
\begin{split}
\sup_{[-T,T]}\|w_{n+1}(t)\|_{C^{\theta}(a,b)}
\leq & \|w_0\|_{C^{\theta}(a,b)} \\
&+ T \sup_{[-T,T]} \|w_{n}(t)\|_{\infty}
\sup_{[-T,T]}\|w_{n}(t)\|_{C^{\theta}(a,b)}\\
&+\sup_{[-T,T]}\|q(t)\|_{C^{\theta}(\R)},
\end{split}
\]
and
\[
\begin{split}
\sup_{[-T,T]}\|&(w_{n+1}-w_n)(t)\|_{C^{\theta}(a,b)}\\
&\leq  T (\sup_{[-T,T]} \|w_{n}(t)\|_{\infty}+\sup_{[-T,T]}\|w_{n-1}(t)\|_{\infty})\times\\
&\hskip25pt \sup_{[-T,T]}\|(w_{n}-w_{n-1})(t)\|_{C^{\theta}(a,b)},
\end{split}
\]
with
$$
T\sup_{[-T,T]}\|w_n(t)\|_{\infty}\leq cT\|w_0\|_{\infty}\leq cT\|u_0\|_X<1/2.
$$

Hence,
\[
w\in C([-T,T]:C^{\theta}(a,b)).
\]

Since
\[
\partial_x u(x,t)=w(s(x,t),t),
\]
with $s=s(x,t)$ as above, see \eqref{aa6a}, it follows that
\[
\partial_xu\in C([-T,T]: C^{\theta}(x(a,t),x(b,t)))
\]
which yields the desired result.

\vskip.1in

 \underline{Case:} $\partial_xu_0|_{(a,b)}\in W^{1,\infty}(a,b)$.
\vskip.1in

Combining the hypothesis $\partial_s w_0(s) \in C([-T,T]:L^{\infty}(a,b))$ with \eqref{aa2}--\eqref{aa3} it follows that
$$
\partial_s w_1(s,t)=\partial_sw_0(s)+\partial_sq(s,t)\in C([-T,T]:L^{\infty}(a,b)).
$$
Also, since for each $n\in\Z^+, n\geq 2$
\begin{equation}\label{aa4a}
\partial_sw_{n+1}(s,t)=\partial_sw_0(s)+2\int_0^t w_n\partial_sw_n(s,\tau)d\tau +\partial_sq(s,t),
\end{equation}
it follows, using the previous steps,  that $\partial_xw_n \in C([-T,T]:L^{\infty}(a,b))$ with $\partial_sw_n$ converging in the $L^{\infty}((a,b)\times [-T,T])$-norm.
Hence,
\begin{equation}\label{aa11}
\partial_s w\in C([-T,T]:L^{\infty}(a,b)).
\end{equation}

Since
\[
\partial_x u(x,t)=w(s(x,t),t),
\]
with $s=s(x,t)$ as above  it follows that
\begin{equation}\label{aa13}
\partial_x^2u\in C([-T,T]: L^{\infty}(x(a,t),x(b,t)))
\end{equation}
which yields the desired result.

\vskip.1in
 \underline{Case:} $\,\partial_xu_0|_{(a,b)}\in C^{1}(a,b)$.
\vskip.1in
Now we have that
\begin{equation} \label{aa14}
 \partial_s q(s,t)=\int_0^t\int_{-\infty}^{\infty} \partial_xG(x(s,\tau)-y)w^2(s(y,\tau),\tau) \partial_sx(s,\tau) dy\,d\tau,
 \end{equation}
with
 $$
 \partial_sx(s,t)=\exp\left(\int_0^tw(s,\tau)d\tau\right).
 $$
 Thus, from  the previous step \eqref{aa11}
 $$
\partial_sx \in C^1([-T,T]: W^{1,\infty}(a,b)),
$$
and  by \eqref{aa14}
\begin{equation}\label{aaa1}
q\in C^1([-T,T]: W^{2,\infty}(a,b)) \hookrightarrow C^1([-T,T]:C^1(a,b)).
\end{equation}

Combining the hypothesis $\partial_s w_0(s) \in C([-T,T]:C(a,b))$ and \eqref{aaa1} one gets that
$$
\partial_s w_1(x,t)=\partial_sw_0(s)+\partial_sq(s,t)\in C([-T,T]:C(a,b)).
$$
By the iteration \eqref{aa4a} $\partial_sw_n \in C([-T,T]:C(a,b))$ for each $n\in\Z^+$ with $\partial_sw_n$. Moreover,
by using the previous step,  $\partial_sw_n$ converges in  $L^{\infty}((a,b)\times [-T,T])$.
Hence,
\begin{equation}
\label{aa18}
\partial_s w\in C([-T,T]:C(a,b)).
\end{equation}

As before we use that
\[
\label{aa19}
\partial_x u(x,t)=w(s(x,t),t),
\]
with $s=s(x,t)$ the inverse function
of 
$$
x(s,t)=s+\int_0^t z(s,\tau)d\tau,
$$
which from \eqref{aa13} satisfies that
$$
x-I_s\in C([-T,T]:W^{2,\infty}(a,b)).
$$
This allows to conclude that
\[
\label{aa20}
\partial_x^2u\in C([-T,T]: C(x(a,t),x(b,t)))
\]
which yields the desired result.

 \vskip.1in
 \underline{Case:} $\partial_xu_0|_{(a,b)}\in C^{1+\theta}(a,b),\;\theta\in(0,1)$.
\vskip.1in

Since from previous step
\begin{equation}\label{ab0}
x(s,t)-s=\int^t_0z(s,\tau)d\tau\in C^1([-T,T]:C^2(a,b)),\
\end{equation}
combining \eqref{aa18} and Lemma \ref{pro2} it follows that for any $\eps>0$
\[
\begin{aligned}
 q(s,t)&=\int_0^t\int_{-\infty}^{\infty} G(x(s,\tau)-y)w^2(s(y,\tau),\tau) dy\,d\tau\\
 &\in C^1([-T,T]:C^2(a+\eps,b-\eps)).
 \end{aligned}
 \]

Thus from the hypothesis $\partial_s w_0(s) \in C([-T,T]:C^{\theta}(a,b))$ one gets that for any $\eps>0$
$$
\partial_s w_1(x,t)=\partial_sw_0(s)+\partial_sq(s,t)\in C([-T,T]:C^{\theta}(a+\eps,b-\eps)).
$$
By the iteration \eqref{aa4a} $\partial_sw_n \in C([-T,T]:C^{\theta}(a+\eps,b-\eps))$ for each $n\in\Z^+$ with $\partial_sw_n$ which using the previous steps converges in the $L^{\infty}([-T,T]:C^{\theta}(a+\eps,b-\eps))$-norm.
Hence,
\[
\partial_s w\in C([-T,T]:C^{\theta}(a+\eps,b-\eps)).
\]

As before we use that
\[
\partial_x u(x,t)=w(s(x,t),t),
\]
with $s=s(x,t)$ the inverse function
of $$
x(s,t)=s+\int_0^t z(s,\tau)d\tau,
$$
which satisfies \eqref{ab0}.
This allows to conclude that
\[
\partial_x^2u\in C([-T,T]: C^{\theta}(x(a+\eps,t),x(b-\eps,t))
\]
for any $\eps>0$ which yields the result, in this case.

 It is clear that the previous argument for $\partial_xu_0|_{(a,b)}\in C^{1+\theta}(a,b)$ works, with slight modifications,  for the cases: 
 $\partial_xu_0|_{(a,b)}\in W^{2,\infty}(a,b)$ and $\partial_xu_0|_{(a,b)}\in C^{2}(a,b)$.
 
 Once these steps have been established, we consider the cases :  $\partial_xu_0|_{(a,b)}\in C^{2+\theta}(a,b)$,
  $\partial_xu_0|_{(a,b)}\in W^{3,\infty}(a,b)$ and  $\partial_xu_0|_{(a,b)}\in C^{3}(a,b)$. From the previous steps one observe that
  for any $\eps>0$ 
  $$
  x(s,t)-s\in C^1([-T,T]: C^3(a+\eps,b-\eps)),
  $$
  and
  $$
  q\in C^1([-T,T]: C^3(a+\eps,b-\eps)).
  $$
  At this point, the argument follows a familiar pattern described in details above.

The general argument is similar so it  will be omitted.

 \end{proof}

\section{Proof of Theorem \ref{unique} and Theorem \ref{decay}}
\begin{proof}[{\bf Proof of Theorem \ref{decay}}]

As in the proof of Theorem \ref{thm2} we consider  $\rho \in C^{\infty}_0(\R)$ such that 
$$
\,\rho(x)\geq 0,\;\;\; \supp(\rho)\subset (-1,1),\;\;\;\int \rho(x)dx =1,
$$
and define $\rho_{\eps}(x)=\frac{1}{\eps} \rho(\frac{x}{\eps})$ for $ \eps\in (0,1)$. For   $u_0\in X$ set
$$
u_0^{\eps}(x)=\rho_{\eps}\ast u_0(x).
$$
Thus, $u_0^{\eps}\in H^s(\R)$ for any $s\in \R$. Let $u^{\eps}=u^{\eps}(x,t)$ the corresponding solution of the IVP associated to the RCH equation \eqref{RCH} provided by Theorem \ref{thm1}.  

As  was mentioned in Remark \ref{rem4} the same argument works for  the CH equation. By Theorem \ref{thm1} there exist $K>0$ and  $T=T(\|u_0\|_X)>0$ (independent of $\eps\in(0,1)$) such that
\[
\sup_{\eps \in (0,1)}\,\sup_{t\in[-T,T]}\|u^{\eps}(t)\|_X\leq 2 c \|u_0\|_X=K.
\]
and
$$
u^{\eps}\in C([-T,T]:H^s(\R)),\;\;\;\;\text{for any}\;\;\;\; s\in \R.
$$

Next,  for each $m\in \Z^+$ we define
\[
\label{weight}
\varphi_m(x)=
\begin{cases}
\begin{aligned}
&\,1,\;\;\;\;\; x\leq 0,\\

&\,e^{\theta x},\;\;\; x\in(0,m],\\

&\,e^{\theta m},\;\;\; x\geq m,

\end{aligned}
\end{cases}
\]
 with  $\theta\in (0,1)$.  Notice that
$$
0\leq \varphi'_m(x)\leq \varphi_m(x).
$$

As in the proof of Theorem B given  in \cite{HMPZ} combining  energy estimate, Gronwall's lemma and the 
fact that if $f\in L^2(\R)\cap L^{\infty}(\R)$, then
$$
\lim_{p\uparrow \infty}\|f\|_p=\|f\|_{\infty},
$$
one finds that for any $t\in [0,T]$ 
\begin{equation}\label{estimate1}
\begin{aligned}
& \|u^{\eps}(t)\varphi_m\|_{\infty} +  \|\partial_xu^{\eps}(t)\varphi_m\|_{\infty} 
\leq e^{TK}\Big(\|u^{\eps}_0\varphi_m\|_{\infty} +  \|\partial_xu_0^{\eps}\varphi_m\|_{\infty} \\
&+ \int_0^T(\|\varphi_m \partial_xG\ast M(u^{\eps}(\tau))\|_{\infty}+
\|\varphi_m \partial^2_xG\ast M(u^{\eps}(\tau))\|_{\infty})d\tau\Big).
\end{aligned}
\end{equation}

 A simple calculation shows that
for any $m\in \Z^+$ one has that
\begin{equation}\label{est1}
\varphi_m(x)\,\int_{-\infty}^{\infty} e^{-|x-y|}\,\frac{1}{\varphi_m(y)}dy \leq \frac{4}{1-\theta}.
\end{equation}
Hence, using \eqref{est1} as in \cite{HMPZ} one gets that
\begin{equation}\label{est2}
|\varphi_m\partial_xG\ast f^2(x)|\leq c\| \varphi_mf\|_{\infty}\|f\|_{\infty}.
\end{equation}
Similarly, using that $\partial_x^2G=G-\delta$ one has that
\begin{equation}\label{est3}
|\varphi_m\partial^2_xG\ast f^2(x)|\leq c\| \varphi_mf\|_{\infty}\|f\|_{\infty}.
\end{equation}

Therefore, inserting \eqref{est2}-\eqref{est3} in \eqref{estimate1} and then using \eqref{bound} it follows that for any $t\in [0,T]$
\begin{equation}\label{final1}
\begin{aligned}
 \|u^{\eps}(t)\varphi_m\|_{\infty} +  \|\partial_xu^{\eps}(t)\varphi_m\|_{\infty} 
&\leq e^{TK}\Big(\|u^{\eps}_0\varphi_m\|_{\infty} +  \|\partial_xu_0^{\eps}\varphi_m\|_{\infty} \\
+ \int_0^T(\|\varphi_m & u^{\eps}(\tau)\|_{\infty}+
\|\varphi_m \partial_xu^{\eps}(\tau)\|_{\infty})d\tau\Big).
\end{aligned}
\end{equation}

Taking limit in \eqref{final1} as $m\uparrow \infty$ one has for any $t\in[0,T]$ 
\begin{equation}
\label{final2}
\begin{aligned}
& \|u^{\eps}(t)\,e^{\theta x}\|_{\infty} +  \|\partial_xu^{\eps}(t)\,e^{\theta x} \|_{\infty} \\
&\leq ce^{TK}( \|u^{\eps}_0\,\max\{1;e^{\theta x}\}\|_{\infty} +  \| \partial_xu^{\eps}_0\,\max\{1;e^{\theta x}\}\|_{\infty})\\
&\leq ce^{TK}( \|u_0\,\max\{1;e^{\theta x}\}\|_{\infty} +  \|\partial_xu_0\,\max\{1;e^{\theta x}\}\|_{\infty}).
\end{aligned}
\end{equation}

Finally, taking the limit as $\eps \downarrow 0$ in \eqref{final2} using the continuous dependence in Theorem \ref{thm1} and passing to subsequence for each $t$ we obtain the desired result
\[
\label{final3}
\begin{aligned}
&\sup_{t\in[0,T]}\|u^{\eps}(t)\,e^{\theta x}\|_{\infty} +  \| \partial_xu^{\eps(t)}\,e^{\theta x} \|_{\infty} \\
&\leq ce^{TK}( \|u_0\,\max\{1;e^{\theta x}\}\|_{\infty} +  \|\partial_xu_0\,\max\{1;e^{\theta x}\}\|_{\infty}).
\end{aligned}
\]
\end{proof}

\begin{proof}[{\bf Proof of Theorem \ref{unique}}]
Once that Theorem \ref{decay} is available the proof of Theorem \ref{unique} follows the argument given in \cite{HMPZ} which for a matter of completeness we sketch here.

Integrating the equation
\[
\label{v00}
\partial_tu+u\partial_xu+\partial_xG\ast M(u)=0,
\]
 where $M(u)=u^2+\frac12(\partial_xu)^2$, it follows that
\begin{equation}
\label{v0}
u(x,t_1)-u_0(x)+\int_0^{t_1} u\partial_xu(x,\tau) d\tau +\partial_x G\ast \int_0^{t_1} M(u)(x,\tau) d \tau.
\end{equation}
By hypothesis one has that
\begin{equation*}
\label{v1}
u(x,t_1)-u_0(x)\sim o(e^{-x}) \;\;\;\text{as}\;\;\; x\uparrow \infty,
\end{equation*}
and by combining the hypothesis \eqref{h1} and Theorem \ref{decay} 
\begin{equation*}
\label{v2}
\int_0^{t_1} u\partial_xu(x,\tau) d\tau \sim o(e^{-x}) \;\;\;\text{as}\;\;\; x\uparrow \infty.
\end{equation*}
Defining
\begin{equation*}
\label{v3}
\mu(x)=\int_0^{t_1} M(u)(x,\tau) d \tau,
\end{equation*}
by Theorem \ref{decay} it follows that
\begin{equation}\label{mu}
0\leq \mu(x)\sim  o(e^{-x}) \;\;\;\text{as}\;\;\; x\uparrow \infty.
\end{equation}
Hence,
\begin{equation*}
\label{E_1E_2}
\begin{aligned}
&\partial_xG\ast \mu(x)=-\frac{1}{2}\int^{\infty}_{-\infty} e^{-|x-y|}\sgn(x-y)\mu(y)dy\\
&=-\frac{1}{2}e^{-x}\int_{-\infty}^xe^y\mu(y)dy+\frac{1}{2}e^x\int_x^{\infty}e^{-y}\mu(y)dy\equiv E_1+E_2.
\end{aligned}
\end{equation*}

Using \eqref{mu} one sees that
\begin{equation*}\label{v4}
E_2=o(1) e^x\int_x^{\infty}e^{-2y}dy\sim o(1)e^{-x}\sim o(e^{-x}).
\end{equation*}
Finally, we observe that if  $\mu \not\equiv 0$ one has that
\[
\int_{-\infty}^x e^y \mu(y)dy\geq c_0>0\;\;\;\;\;\;\text{for}\;\;\;\;\;x\gg 1,
\]
which implies that
$$
-E_1\geq \frac{c_0}{2} e^{-x},\;\;\;\;\;\;\text{for}\;\;\;\;\;x\gg 1.
$$
This combined with the previous estimates inserted in \eqref{v0} yields a contradiction. Therefore, $\mu \equiv 0$ and as a result $u\equiv 0$, which is the desired result.

 \end{proof}




\end{document}